\documentclass[a4paper]{amsart}
\usepackage{amssymb,mathrsfs}
\theoremstyle{plain}
\newtheorem{theorem}[equation]{\bf Theorem}
\newtheorem{lemma}[equation]{\bf Lemma}
\newtheorem{proposition}[equation]{\bf Proposition}
\newtheorem{corollary}[equation]{\bf Corollary}
\theoremstyle{definition}
\newtheorem{definition}[equation]{\bf Definition}
\newtheorem{remark}[equation]{\bf Remark}
\newtheorem{example}[equation]{\bf Example}
\numberwithin{equation}{section}
\begin{document}

\title[A topic on homogeneous vector bundles over elliptic orbits]
{A topic on homogeneous vector bundles over elliptic orbits: A condition for the vector spaces of their cross-sections to be finite dimensional}

\author[N.~Boumuki]{Nobutaka Boumuki}
\thanks{This work was supported by JSPS KAKENHI Grant Number JP 17K05229.}
\date{}

\address{
Division of Mathematical Sciences, Faculty of Science and Technology\endgraf 
Oita University, 700 Dannoharu, Oita-shi, Oita 870-1192, JAPAN}
\email{boumuki@oita-u.ac.jp}
\maketitle

\begin{abstract}
   In this paper we consider the complex vector spaces of holomorphic cross-sections of homogeneous holomorphic vector bundles over elliptic adjoint orbits, and provide a sufficient condition for the vector spaces to be finite dimensional in view of root systems. 
\end{abstract}

\section{Introduction}\label{sec-1}
   For a connected real semisimple Lie group $G$, the adjoint orbit $\mathrm{Ad}G(T)=G/C_G(T)$ of $G$ through an elliptic element $T\in\frak{g}$ is called an {\it elliptic} ({\it adjoint}) {\it orbit}. 
   Here an element $T\in\frak{g}$ is said to be {\it elliptic}, if $\mathrm{ad}T$ is a semisimple linear transformation of $\frak{g}$ and all the eigenvalues of $\mathrm{ad}T$ are purely imaginary. 
   It is known that elliptic orbits can be geometrically characterized as follows (cf.\ Dorfmeister-Guan \cite{DoGu1,DoGu2}): 
\begin{quote}
   Any elliptic orbit $G/C_G(T)$ is a homogeneous pseudo-K\"{a}hler manifold of $G$. 
   Conversely, a homogeneous pseudo-K\"{a}hler manifold $M$ of $G$ is an elliptic orbit whenever $G$ acts on $M$ almost effectively.
\end{quote}   
   Accordingly there is no essential difference between elliptic orbits and homogeneous pseudo-K\"{a}hler manifolds of real semisimple Lie groups.
   Let us give examples of elliptic orbits. 
   A complex projective space $CP^n$ is one of the Hermitian symmetric spaces of compact type, any Hermitian symmetric space $G_u/K$ of compact type is one of the complex flag manifolds (which are also called generalized flag manifolds or K\"{a}hler C-spaces), and all complex flag manifolds $G_\mathbb{C}/Q$ are elliptic orbits.
   These are examples of elliptic orbits which are compact.
   As a non-compact example, one knows that all symmetric bounded domains $D$ in $\mathbb{C}^n$ are elliptic orbits.
   In this paper, we deal with such spaces. 
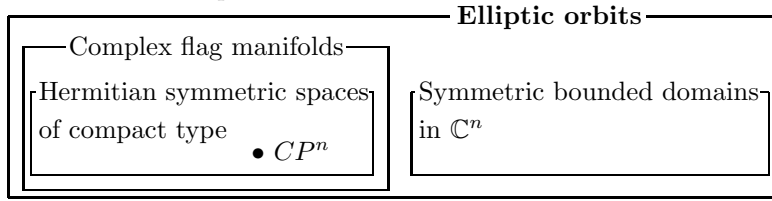
\begin{center}
\unitlength=1mm
\begin{picture}(102,26)
\put(1,1){\line(1,0){102}}
\put(1,1){\line(0,1){24}}
\put(103,1){\line(0,1){24}}
\put(60,24){{\bf Elliptic orbits}}
\put(1,25){\line(1,0){58}}
\put(103,25){\line(-1,0){18}}
\put(4,4){\line(1,0){45}}
\put(5,14){Hermitian symmetric spaces}
\put(5,9){of compact type}
\put(4,4){\line(0,1){11}}
\put(49,4){\line(0,1){11}}
\put(33,6){$\bullet$ $CP^n$}
\put(4,15){\line(1,0){0.6}}
\put(49,15){\line(-1,0){0.5}}
\put(3,2){\line(1,0){48}}
\put(3,2){\line(0,1){19}}
\put(9,20){Complex flag manifolds}
\put(3,21){\line(1,0){5.5}}
\put(51,2){\line(0,1){19}}
\put(51,21){\line(-1,0){5.5}}
\put(55,14){Symmetric bounded domains}
\put(55,9){in $\mathbb{C}^n$}
\put(54,4){\line(1,0){47}}
\put(54,4){\line(0,1){11}}
\put(101,4){\line(0,1){11}}
\put(54,15){\line(1,0){0.4}}
\put(101,15){\line(-1,0){1}}
\end{picture}
\end{center}
   Now, let us explain our research background.
   Let $G_\mathbb{C}$ be a connected complex semisimple Lie group, let $G$ be a connected closed subgroup of $G_\mathbb{C}$ such that $\frak{g}$ is a real form of $\frak{g}_\mathbb{C}$, and let $T$ be a non-zero elliptic element of $\frak{g}$.
   Setting 
\[
\begin{array}{l}
   L:=C_G(T),\qquad  \mbox{$\frak{g}^\lambda:=\{X\in\frak{g}_\mathbb{C} \,|\, \mathrm{ad}T(X)=i\lambda X\}$ for $\lambda\in\mathbb{R}$},\\
   Q^-:=\{x\in G_\mathbb{C} \,|\, \mathrm{Ad}x\bigl(\bigoplus_{\mu\leq 0}\frak{g}^\mu\bigr)\subset\bigoplus_{\mu\leq 0}\frak{g}^\mu\},
\end{array}
\]
one has an elliptic orbit $G/L$, a complex flag manifold $G_\mathbb{C}/Q^-$ and $L=G\cap Q^-$; besides, it turns out that $\iota:G/L\to G_\mathbb{C}/Q^-$, $gL\mapsto gQ^-$, is a $G$-equivariant real analytic embedding whose image is a simply connected domain in $G_\mathbb{C}/Q^-$, and that $GQ^-$ is a domain in $G_\mathbb{C}$.
   Henceforth, we assume $G/L$ to be a domain in $G_\mathbb{C}/Q^-$ and it to be a homogeneous complex manifold of $G$ via this $\iota$.
\begin{center}
\unitlength=1mm
\begin{picture}(72,20)
\put(10,1){$G/L$}
\put(37,4){$\iota$}
\put(20,3){\vector(1,0){40}}
\put(2,17){$\iota^\sharp(G_\mathbb{C}\times_\rho{\sf V})$}
\put(12,14){\vector(0,-1){8}}
\put(62,1){$G_\mathbb{C}/Q^-$}
\put(59,17){$G_\mathbb{C}\times_\rho{\sf V}$}
\put(67,14){\vector(0,-1){8}}
\end{picture}
\end{center}    
   In addition, let ${\sf V}$ be a finite dimensional complex vector space and let $\rho:Q^-\to GL({\sf V})$, $q\mapsto\rho(q)$, be a holomorphic homomorphism. 
   Denote by $G_\mathbb{C}\times_\rho{\sf V}$ the fiber bundle over the complex flag manifold $G_\mathbb{C}/Q^-$, with standard fiber ${\sf V}$ and structure group $Q^-$, which is associated to the principal fiber bundle $\pi_\mathbb{C}:G_\mathbb{C}\to G_\mathbb{C}/Q^-$, $x\mapsto xQ^-$, and denote by $\iota^\sharp(G_\mathbb{C}\times_\rho{\sf V})$ the restriction of the bundle $G_\mathbb{C}\times_\rho{\sf V}$ to the domain $G/L\subset G_\mathbb{C}/Q^-$.
   Then one may assume that 
\[
\begin{split}
&  \mathcal{V}_{G_\mathbb{C}/Q^-}\!\!
   :=\!\!\left\{\begin{array}{@{}l@{\,\,}|l@{}}
   h:G_\mathbb{C}\to{\sf V} 
   & \begin{array}{@{\!}l@{\!}}
      \mbox{(1) $h$ is holomorphic},\\
      \mbox{(2) $h(xq)=\rho(q)^{-1}(h(x))$ for all $(x,q)\in G_\mathbb{C}\!\times\! Q^-$}
     \end{array}\end{array}\right\}\! \mbox{ and}\\
&  \mathcal{V}_{G/L}\!\!
   :=\!\!\left\{\begin{array}{@{}l@{\,\,}|l@{}}
   \psi:GQ^-\to{\sf V} 
   & \begin{array}{@{\!}l@{\!}}
      \mbox{(1) $\psi$ is holomorphic},\\
      \mbox{(2) $\psi(yq)=\rho(q)^{-1}(\psi(y))$ for all $(y,q)\in GQ^-\!\times\! Q^-$}
     \end{array}\end{array}\right\}
\end{split} 
\]
are the complex vector spaces of holomorphic cross-sections of the bundles $G_\mathbb{C}\times_\rho{\sf V}$ and $\iota^\sharp(G_\mathbb{C}\times_\rho{\sf V})$, respectively.
   Here, we remark that the vector space $\mathcal{V}_{G_\mathbb{C}/Q^-}$ is always finite dimensional,
\[
   \dim_\mathbb{C}\mathcal{V}_{G_\mathbb{C}/Q^-}<\infty
\]
because $G_\mathbb{C}/Q^-$ is a connected compact complex manifold; but, in contrast, $\mathcal{V}_{G/L}$ is not necessarily finite dimensional---for example, $\dim_\mathbb{C}\mathcal{V}_{G/L}=\infty$ in the case where $G/L$ is a symmetric bounded domain in $\mathbb{C}^n$ and $\mathcal{V}_{G/L}$ is the vector space $\mathcal{O}(T^{1,0}(G/L))$ of holomorphic vector fields on it. 
   This poses us the following problem:
\begin{center}
   ``What is a condition for $\dim_\mathbb{C}\mathcal{V}_{G/L}<\infty$ ?'' 
\end{center}
   In this paper we partially solve this problem.
\par

   The main purpose of this paper is to provide a sufficient condition so that all the holomorphic mappings $\psi\in\mathcal{V}_{G/L}$ can be continued analytically from $GQ^-$ to $G_\mathbb{C}$. 
   In view of a root system $\triangle$ of $\frak{g}_\mathbb{C}$, we assert the following statement (see Subsection \ref{subsec-3.1}, Theorem \ref{thm-3.1}): 
\begin{quote}
   {\it Suppose that {\rm (S)} there exists a fundamental root system $\Pi_\triangle$ of $\triangle$ satisfying  
   \begin{enumerate}
   \item[]{\rm (s1)} $\alpha(-iT)\geq 0$ for all $\alpha\in\Pi_\triangle$, and
   \item[]{\rm (s2)} $\frak{g}_\beta\subset\frak{k}_\mathbb{C}$ for every $\beta\in\Pi_\triangle$ with $\beta(T)\neq0$.
   \end{enumerate}
   Then, all the holomorphic mappings $\psi\in\mathcal{V}_{G/L}$ extend uniquely to holomorphic ones $\hat{\psi}\in\mathcal{V}_{G_\mathbb{C}/Q^-}$ and $\dim_\mathbb{C}\mathcal{V}_{G/L}=\dim_\mathbb{C}\mathcal{V}_{G_\mathbb{C}/Q^-}\hspace{-0.65mm}<\hspace{-0.65mm}\infty$}.
\end{quote} 
   Pay attention to that in the case where the above supposition (S) holds, the vector space $\mathcal{V}_{G/L}$ is finite dimensional for any complex vector space ${\sf V}$ of $\dim_\mathbb{C}{\sf V}<\infty$ and any holomorphic homomorphism $\rho:Q^-\to GL({\sf V})$. 
   Hence, in particular, one can deduce that in this case, the group $\mathrm{Hol}(G/L)$ of holomorphic automorphisms of $G/L$ is a (finite dimensional) Lie group.\par
   
   This paper consists of four sections. 
   In Section \ref{sec-2} we mainly review known facts about elliptic orbits, generalized Bruhat decompositions and homogeneous holomorphic vector bundles.
   In Section \ref{sec-3} we state the main result in this paper (Theorem \ref{thm-3.1}) and demonstrate it by taking a continuous representation $\varrho$ of $G$ on $\mathcal{V}_{G/L}$, a generalized Bruhat decomposition of $G_\mathbb{C}$ and the second Riemann removable singularity theorem into account. 
   Finally in Section \ref{sec-4}, we give some examples which satisfy the supposition (S) in Theorem \ref{thm-3.1}, and give an example which does not so.
   We will see that the (S) cannot hold for any symmetric bounded domain $D$ in $\mathbb{C}^n$, cf.\ Example \ref{ex-4.2}.

\section{Preliminaries}\label{sec-2}
   In this section we first fix the notation utilized in this paper, and afterwards review known facts about elliptic orbits, generalized Bruhat decompositions and homogeneous holomorphic vector bundles. 
   We will give two Lemmas \ref{lem-2.6} and \ref{lem-2.14}, Corollary \ref{cor-2.20} and Proposition \ref{prop-2.27} especially needed in Section \ref{sec-3}.
\subsection{Notation}\label{subsec-2.1}
   Throughout this paper, for a Lie group $G$, we denote its Lie algebra by the corresponding Fraktur small letter $\frak{g}$, and utilize the following notation: 
\begin{enumerate} 
\item[](n1)
   $i:=\sqrt{-1}$,
\item[](n2)
   $\mathrm{Ad}$, $\mathrm{ad}$ : the adjoint representation of $G$, $\frak{g}$, 
\item[](n3)
   $C_G(T):=\{g\in G \,|\,\mathrm{Ad}g(T)=T\}$ for an element $T\in\frak{g}$, 
\item[](n4)
   $N_G(\frak{m}):=\{g\in G \,|\,\mathrm{Ad}g(\frak{m})\subset\frak{m}\}$ for a vector subspace $\frak{m}\subset\frak{g}$, 
\item[](n5)
   $\frak{m}\oplus\frak{n}$ : the direct sum of vector spaces $\frak{m}$ and $\frak{n}$,
\item[](n6) 
   $GL(V)$ : the general linear group on a complex vector space $V$.
\end{enumerate}
   Besides, we sometimes denote by $f|_A$ the restriction of a mapping $f$ to a set $A$.

\subsection{Elliptic orbits}\label{subsec-2.2}
   Kobayashi \cite{Ko} has introduced the notion of elliptic orbit, which is as follows: 
\begin{definition}[cf.\ Kobayashi {\cite[p.5]{Ko}}] 
   Let $\frak{g}$ be a real semisimple Lie algebra and $G$ a connected Lie group with Lie algebra $\frak{g}$. 
   An element $T\in\frak{g}$ is said to be {\it elliptic}, if $\mathrm{ad}T$ is a semisimple linear transformation of $\frak{g}$ and all the eigenvalues of $\mathrm{ad}T$ are purely imaginary. 
   The adjoint orbit $\mathrm{Ad}G(T)=G/C_G(T)$ of $G$ through an elliptic element $T\in\frak{g}$ is called an {\it elliptic} ({\it adjoint}) {\it orbit}. 
\end{definition}

   Now, let $G_\mathbb{C}$ be a connected complex semisimple Lie group, let $G$ be a connected closed subgroup of $G_\mathbb{C}$ such that $\frak{g}$ is a real form of $\frak{g}_\mathbb{C}$, and let $T$ be a non-zero elliptic element of $\frak{g}$. 
   Then we set 
\begin{equation}\label{eq-2.2}
\left\{
\begin{array}{@{}l}
   \begin{array}{@{}ll}
   L:=C_G(T), & L_\mathbb{C}:=C_{G_\mathbb{C}}(T),
   \end{array}\\
   \mbox{$\frak{g}^\lambda:=\{X\in\frak{g}_\mathbb{C} \,|\, \mathrm{ad}T(X)=i\lambda X\}$ for $\lambda\in\mathbb{R}$},\\
   \begin{array}{@{}lll}
   \frak{u}^\pm:=\bigoplus_{\lambda>0}\frak{g}^{\pm\lambda}, 
   & U^\pm:=\exp\frak{u}^\pm, 
   & Q^\pm:=N_{G_\mathbb{C}}(\frak{l}_\mathbb{C}\oplus\frak{u}^\pm),
   \end{array}
\end{array}\right.
\end{equation} 
where $\frak{g}^\lambda=\{0\}$ in the case where $\lambda$ is different from the eigenvalues of $\mathrm{ad}T$, we denote by $\exp:\frak{g}_\mathbb{C}\to  G_\mathbb{C}$ the exponential mapping, and $\frak{u}^\pm_T$ will stand for the above $\frak{u}^\pm$ for once in Lemma \ref{lem-2.6}. 
   Since $T\in\frak{g}$ is elliptic, there exists a Cartan decomposition $\frak{g}=\frak{k}\oplus\frak{p}$ of $\frak{g}$ such that 
\begin{equation}\label{eq-2.3}
   T\in\frak{k},
\end{equation} 
where $\frak{k}$ is a maximal compact subalgebra of $\frak{g}$. 
   Noting that the center $Z(G)$ of $G$ is finite due to $Z(G)\subset Z(G_\mathbb{C})$ and that $\frak{g}_u:=\frak{k}\oplus i\frak{p}$ is a compact real form of $\frak{g}_\mathbb{C}$, we denote by $K$ and $G_u$ the maximal compact subgroups of $G$ and $G_\mathbb{C}$ corresponding to the subalgebras $\frak{k}\subset\frak{g}$ and $\frak{g}_u\subset\frak{g}_\mathbb{C}$, respectively. 
   In addition, we denote by the (anti-holomorphic) Cartan involution $\bar{\theta}$ of $G_\mathbb{C}$ such that 
\begin{equation}\label{eq-2.4}
   G_u=\{g_u\in G_\mathbb{C} \,|\, \bar{\theta}(g_u)=g_u\}.
\end{equation} 
   Let us give easy lemmas and review a known fact. 
\begin{lemma}\label{lem-2.5}
   In the setting \eqref{eq-2.2}$;$ 
\begin{enumerate}
\item[{\rm (1)}]
   $L_\mathbb{C}$ is a connected closed complex subgroup of $G_\mathbb{C}$ with $\frak{l}_\mathbb{C}=\frak{g}^0$. 
\item[{\rm (2)}] 
   $\frak{g}_\mathbb{C}=\bigoplus_{\lambda\in\mathbb{R}}\frak{g}^\lambda=\frak{u}^+\oplus\frak{l}_\mathbb{C}\oplus\frak{u}^-$.  
\item[{\rm (3)}]
   $\frak{l}_\mathbb{C}\oplus\frak{u}^+=\bigoplus_{\mu\geq 0}\frak{g}^\mu$ and $\frak{l}_\mathbb{C}\oplus\frak{u}^-=\bigoplus_{\mu\leq 0}\frak{g}^\mu$. 
\item[{\rm (4)}] 
   $\mathrm{Ad}L_\mathbb{C}(\frak{g}^\lambda)\subset\frak{g}^\lambda$ for all $\lambda\in\mathbb{R}$. 
\item[{\rm (5)}]
   $[\frak{g}^\lambda,\frak{g}^\mu]\subset\frak{g}^{\lambda+\mu}$ for all $\lambda,\mu\in\mathbb{R}$. 
\item[{\rm (6)}]
   $\frak{u}^s$ is a complex nilpotent subalgebra of $\frak{g}_\mathbb{C}$ such that $\mathrm{Ad}L_\mathbb{C}(\frak{u}^s)\subset\frak{u}^s$, for each $s=\pm$.
\end{enumerate} 
   In the setting \eqref{eq-2.2}, \eqref{eq-2.3} and \eqref{eq-2.4}$;$   
\begin{enumerate}
\item[{\rm (7)}]
   $\bar{\theta}_*(\frak{g}^\lambda)=\frak{g}^{-\lambda}$ for all $\lambda\in\mathbb{R}$.  
\item[{\rm (8)}]
   $\bar{\theta}(L_\mathbb{C})=L_\mathbb{C}$, $\bar{\theta}_*(\frak{u}^+)=\frak{u}^-$, $\bar{\theta}_*(\frak{u}^-)=\frak{u}^+$. 
\end{enumerate} 
\end{lemma}   
\begin{proof}
   Since $G_\mathbb{C}$ is connected semisimple and $T\in\frak{g}$ is an elliptic element of $\frak{g}_\mathbb{C}$ also, one shows that $L_\mathbb{C}=C_{G_\mathbb{C}}(T)$ is connected.
   The rest of proof is trivial.
\end{proof}

\begin{lemma}[cf.\ \cite{BoNo}]\label{lem-2.6}
   Let $G_\mathbb{C}$ be a connected complex semisimple Lie group, let $G$ be a connected closed subgroup of $G_\mathbb{C}$ such that $\frak{g}$ is a real form of $\frak{g}_\mathbb{C}$, and let $T$ be a non-zero elliptic element of $\frak{g}$. 
   Fix a Cartan decomposition $\frak{g}=\frak{k}\oplus\frak{p}$ with $T\in\frak{k}$, and take a maximal torus $i\frak{h}_\mathbb{R}$ of $\frak{g}_u=\frak{k}\oplus i\frak{p}$ containing $T$. 
   Then, there exists an elliptic element $T'\in\frak{g}$ such that 
\begin{enumerate}
\item[{\rm (i)}]   
   all the eigenvalues of $\mathrm{ad}iT'$ are integer,
\item[{\rm (ii)}] 
   $C_G(T)=C_G(T')$, 
\item[{\rm (iii)}]   
   $\frak{u}^+_T=\frak{u}^+_{T'}$, $\frak{u}^-_T=\frak{u}^-_{T'}$ and
\item[{\rm (iv)}] 
   $T'\in i\frak{h}_\mathbb{R}$.
\end{enumerate}
   Here, we refer to \eqref{eq-2.2} for $\frak{u}^\pm_T$, $\frak{u}^\pm_{T'}$.
\end{lemma}   
\begin{proof}
   One can conclude this lemma by the proof of Theorem 2.3 in \cite[p.66]{BoNo}.
\end{proof}

   From Lemma \ref{lem-2.5} we deduce  
\begin{proposition}\label{prop-2.7}
   In the setting \eqref{eq-2.2}$;$
\begin{enumerate}
\item[{\rm (1)}]
   $U^s$ is a simply connected, closed complex nilpotent subgroup of $G_\mathbb{C}$ whose Lie algebra coincides with $\frak{u}^s$, and $\exp:\frak{u}^s\to U^s$ is biholomorphic, for each $s=\pm$.
\item[{\rm (2)}]
   $Q^s$ is a connected, closed complex parabolic subgroup of $G_\mathbb{C}$ such that 
$Q^s=L_\mathbb{C}\ltimes U^s$ 
$($semidirect$)$ and $\frak{q}^s=(\frak{l}_\mathbb{C}\oplus\frak{u}^s)=\bigoplus_{\mu\geq 0}\frak{g}^{s\mu}$, for each $s=\pm$. 
\item[{\rm (3)}]
   $U^+\times Q^-\ni(u,q)\mapsto uq\in G_\mathbb{C}$ is a holomorphic embedding whose image is a dense, domain in $G_\mathbb{C}$.
\item[{\rm (4)}] 
   $L$ is a connected closed subgroup of $G$, and the homogeneous space $G/L$ is simply connected. 
\item[{\rm (5)}]
   $L=G\cap Q^-$.   
\item[{\rm (6)}]
   $\iota:G/L\to G_\mathbb{C}/Q^-$, $gL\mapsto gQ^-$, is a $G$-equivariant real analytic embedding whose image is a simply connected domain in $G_\mathbb{C}/Q^-$.
\item[{\rm (7)}]
   $GQ^-$ is a domain in $G_\mathbb{C}$.
\end{enumerate}
   In the setting \eqref{eq-2.2}, \eqref{eq-2.3} and \eqref{eq-2.4}$;$
\begin{enumerate}
\item[{\rm (8)}]
   $\bar{\theta}(U^+)=U^-$, $\bar{\theta}(U^-)=U^+$, $\bar{\theta}(Q^+)=Q^-$, $\bar{\theta}(Q^-)=Q^+$ and $\bar{\theta}(L)=L$.
\end{enumerate}
\end{proposition}
\begin{proof}
   e.g.\ Warner \cite{Wa} or \cite[Paragraph 2.4.2]{BoNo}.
\end{proof}

\begin{remark}\label{rem-2.8}
   In general, there are several kinds of invariant complex structures on the elliptic orbit $G/L$. 
   In this paper we deal with the complex structure on $G/L$ induced by $\iota:G/L\to G_\mathbb{C}/Q^-$, $gL\mapsto gQ^-$. 
   Here, the imaginary unit $i\in\mathbb{C}$ gives rise to a $G_\mathbb{C}$-invariant complex structure $J$ on the complex flag manifold $G_\mathbb{C}/Q^-$ in a natural way.
\end{remark}

   Proposition \ref{prop-2.7}-(3), (7) leads to 
\begin{corollary}\label{cor-2.9}
   In the setting \eqref{eq-2.2}$;$ the following two items hold for given finite elements $x_1,x_2,\dots,x_j\in G_\mathbb{C}:$
\begin{enumerate}
\item[{\rm (1)}]
   The intersection $GQ^-\cap x_1U^+Q^-\cap\cdots\cap x_jU^+Q^-$ is a non-empty open subset of $G_\mathbb{C}$. 
\item[{\rm (2)}]
   The union $GQ^-\cup x_1U^+Q^-\cup\cdots\cup x_jU^+Q^-$ is a dense, domain in $G_\mathbb{C}$.
\end{enumerate}
\end{corollary}   

\subsection{Root systems and generalized Bruhat decompositions}\label{subsec-2.3}
   We review fundamental results about root systems and modify a generalized Bruhat decomposition of $G_\mathbb{C}$ for our situation (see Proposition \ref{prop-2.18}-(3)). 
   The setting \eqref{eq-2.2}, \eqref{eq-2.3} and \eqref{eq-2.4} remains valid in this subsection.
\subsubsection{Root systems and Weyl groups}\label{subsec-2.3.1}
   Let $i\frak{h}_\mathbb{R}$ be a maximal torus of $\frak{g}_u=\frak{k}\oplus i\frak{p}$ containing the element $T$, let $\triangle=\triangle(\frak{g}_\mathbb{C},\frak{h}_\mathbb{C})$ be the (non-zero) root system of $\frak{g}_\mathbb{C}$ relative to $\frak{h}_\mathbb{C}$, where $\frak{h}_\mathbb{C}$ is the complex vector subspace of $\frak{g}_\mathbb{C}$ generated by $i\frak{h}_\mathbb{R}$, and let $\frak{g}_\alpha$ be the root subspace of $\frak{g}_\mathbb{C}$ for $\alpha\in\triangle$. 
   For each root $\alpha\in\triangle$, there exists a unique $H_\alpha\in\frak{h}_\mathbb{C}$ such that $\alpha(H)=B_{\frak{g}_\mathbb{C}}(H_\alpha,H)$ for all $H\in\frak{h}_\mathbb{C}$, where $B_{\frak{g}_\mathbb{C}}$ is the Killing form of $\frak{g}_\mathbb{C}$. 
   Then $\frak{h}_\mathbb{R}=\mathrm{span}_\mathbb{R}\{H_\alpha\,|\,\alpha\in\triangle\}$, and for every $\alpha\in\triangle$ there exists a vector $E_\alpha\in\frak{g}_\alpha$ satisfying 
\begin{equation}\label{eq-2.10}
   \mbox{$(E_\alpha-E_{-\alpha}), i(E_\alpha+E_{-\alpha})\in\frak{g}_u$ and $[E_\alpha,E_{-\alpha}]=(2/\alpha(H_\alpha))H_\alpha$}
\end{equation}
(cf.\ Helgason \cite[Lemma 3.1, p.257--258]{He}). 
   Here, it is immediate from \eqref{eq-2.4} that $\frak{g}_u=i\frak{h}_\mathbb{R}\oplus\bigoplus_{\alpha\in\triangle}\mathrm{span}_\mathbb{R}\{E_\alpha-E_{-\alpha}\}\oplus\mathrm{span}_\mathbb{R}\{i(E_\alpha+E_{-\alpha})\}$, and 
\begin{equation}\label{eq-2.11}
   \mbox{$\bar{\theta}_*(E_\alpha)=-E_{-\alpha}$ for all $\alpha\in\triangle$}.
\end{equation}
   Define a Weyl group $\mathscr{W}$ of $G_\mathbb{C}$ and an action $\zeta$ of $\mathscr{W}$ on the dual space $(\frak{h}_\mathbb{C})^*$ by
\begin{equation}\label{eq-2.12}
\left\{
\begin{array}{@{}l}
   \mathscr{W}:=N_{G_u}(i\frak{h}_\mathbb{R})/C_{G_u}(i\frak{h}_\mathbb{R}),\\ 
   \mbox{$\zeta([w])\eta:={}^t\!\mathrm{Ad}w^{-1}(\eta)$ for $[w]\in\mathscr{W}$ and $\eta\in(\frak{h}_\mathbb{C})^*$},
\end{array}\right. 
\end{equation}
where $[w]$ stands for the left coset $wC_{G_u}(i\frak{h}_\mathbb{R})$. 
   By use of $E_\alpha$ in \eqref{eq-2.10} we set
\begin{equation}\label{eq-2.13}
   \mbox{$w_\alpha:=\exp(\pi/2)(E_\alpha-E_{-\alpha})$ for $\alpha\in\triangle$}.
\end{equation}
   Needless to say, $w_\alpha$ belongs to $N_{G_u}(i\frak{h}_\mathbb{R})$ and so $[w_\alpha]\in\mathscr{W}$ for every root $\alpha\in\triangle$; besides, $\zeta([w_\alpha])$ is the reflection along $\alpha$ which leaves $\triangle$ invariant.
   We need 
\begin{lemma}\label{lem-2.14}
   Let $\frak{k}_\mathbb{C}$ be the complex subalgebra of $\frak{g}_\mathbb{C}$ generated by $\frak{k}$. 
   For a root $\beta\in\triangle=\triangle(\frak{g}_\mathbb{C},\frak{h}_\mathbb{C})$ with $\beta(T)\neq 0$, the following {\rm (a)}, {\rm (b)} and {\rm (c)} are equivalent$:$
\[
\begin{array}{lll}
   \mbox{{\rm (a)} $\frak{g}_\beta\subset\frak{k}_\mathbb{C}$}, & \mbox{{\rm (b)} $E_\beta\in\frak{k}_\mathbb{C}$}, & \mbox{{\rm (c)} $(E_\beta-E_{-\beta})\in\frak{k}$}.
\end{array}
\]
   Therefore, $w_\beta=\exp(\pi/2)(E_\beta-E_{-\beta})$ belongs to $K\cap N_{G_u}(i\frak{h}_\mathbb{R})$ whenever one of the {\rm (a)}, {\rm (b)} and {\rm (c)} holds. 
\end{lemma} 
\begin{proof}
   Since (a) $\Leftrightarrow$ (b) is obvious, we only confirm (b) $\Leftrightarrow$ (c).\par

   (b) $\Rightarrow$ (c): This follows by \eqref{eq-2.11}, $\bar{\theta}_*(\frak{k}_\mathbb{C})\subset\frak{k}_\mathbb{C}$ and $\frak{k}=\{Y\in\frak{k}_\mathbb{C} \,|\, \bar{\theta}_*(Y)=Y\}$.\par
   
   (c) $\Rightarrow$ (b): 
   Suppose that $(E_\beta-E_{-\beta})\in\frak{k}$. 
   Then, from \eqref{eq-2.3} one obtains 
\[
   \beta(T)(E_\beta+E_{-\beta})=[T,E_\beta-E_{-\beta}]\in[\frak{k},\frak{k}]\subset\frak{k};
\]
and so $0\neq\beta(T)\in i\mathbb{R}$ yields $(E_\beta+E_{-\beta})\in i\frak{k}$. 
   Hence $E_\beta=(1/2)(E_\beta-E_{-\beta}+E_\beta+E_{-\beta})\in\frak{k}+i\frak{k}\subset\frak{k}_\mathbb{C}$.  
\end{proof}

\subsubsection{Generalized Bruhat decompositions}\label{subsec-2.3.2}
   We continue to obey the setting of Paragraph \ref{subsec-2.3.1}.\par

   Our first aim in this paragraph is to state Proposition \ref{prop-2.17} which is a result of Kostant \cite{Ko1,Ko2} and the second one is to modify a generalized Bruhat decomposition of $G_\mathbb{C}$ for our situation. 
   For the aim, we are going to fix two Iwasawa decompositions of $G_\mathbb{C}$ first.\par
 
   Let $\Pi_\triangle$ be a fundamental root system\footnote{There is such a system with (s1)---for example, consider the lexicographic linear ordering on the dual space $(\frak{h}_\mathbb{R})^*$ associated with a real base $-iT=:A_1,A_2,\dots,A_\ell$ of $\frak{h}_\mathbb{R}$.} of $\triangle=\triangle(\frak{g}_\mathbb{C},\frak{h}_\mathbb{C})$ satisfying  
\[
   \mbox{(s1) $\alpha(-iT)\geq 0$ for all $\alpha\in\Pi_\triangle$}. 
\]
   Relative to this $\Pi_\triangle$ we fix the set $\triangle^+$ of positive roots, and put $\triangle^-:=-\triangle^+$. 
   Then (s1) yields $\beta(-iT)\geq 0$ for all $\beta\in\triangle^+$.
   Setting $\frak{n}^s:=\bigoplus_{\beta\in\triangle^s}\frak{g}_\beta$ and $\frak{b}^s:=\frak{h}_\mathbb{C}\oplus\frak{n}^s$ ($s=\pm$) one has Iwasawa decompositions $\frak{g}_\mathbb{C}=\frak{g}_u\oplus\frak{h}_\mathbb{R}\oplus\frak{n}^\pm$ of $\frak{g}_\mathbb{C}$; moreover, it follows from \eqref{eq-2.2} and $\frak{g}_\mathbb{C}=\frak{n}^+\oplus\frak{h}_\mathbb{C}\oplus\frak{n}^-$ that 
\begin{equation}\label{eq-2.15}
\left\{
\begin{array}{@{}l}
   \frak{u}^+=\bigoplus_{\lambda>0}\frak{g}^\lambda\subset\bigoplus_{\beta\in\triangle^+}\frak{g}_\beta=\frak{n}^+\subset\frak{b}^+\subset\bigoplus_{\mu\geq 0}\frak{g}^\mu=(\frak{l}_\mathbb{C}\oplus\frak{u}^+)=\frak{q}^+,\\
   \frak{u}^-\subset\frak{n}^-\subset\frak{b}^-\subset\frak{q}^-.  
\end{array}\right. 
\end{equation}
   Denote by $G_\mathbb{C}=G_uH_\mathbb{R}N^\pm$ the Iwasawa decompositions of $G_\mathbb{C}$ corresponding to the $\frak{g}_\mathbb{C}=\frak{g}_u\oplus\frak{h}_\mathbb{R}\oplus\frak{n}^\pm$, respectively.\par

   Following Kostant \cite{Ko1,Ko2}, we set 
\begin{equation}\label{eq-2.16}
\left\{
\begin{array}{@{}l}
   \triangle(\frak{u}^\pm):=\{\beta\in\triangle^\pm \,|\, \beta(T)\neq0\}\,\,\bigl(=\{\alpha\in\triangle \,|\, \pm\alpha(-iT)>0\}\bigr),\\
   \begin{array}{@{}ll}
   \triangle(\frak{l}_\mathbb{C}):=\{\gamma\in\triangle \,|\, \gamma(T)=0\}, \triangle^\pm(\frak{l}_\mathbb{C}):=\triangle(\frak{l}_\mathbb{C})\cap\triangle^\pm,
   \end{array}\\
   \mbox{$\Phi_{[w]}:=\{\beta\in\triangle^+ \,|\, \zeta([w])^{-1}\beta\in\triangle^-\}$ for $[w]\in\mathscr{W}$},\\
   \begin{array}{@{}ll}
   \mathscr{W}^1:=\{[\sigma]\in\mathscr{W} \,|\, \Phi_{[\sigma]}\subset\triangle(\frak{u}^+)\}, 
   & \mathscr{W}_1:=N_{L_u}(i\frak{h}_\mathbb{R})/C_{L_u}(i\frak{h}_\mathbb{R}),
   \end{array}
\end{array}\right.
\end{equation}
where $L_u:=C_{G_u}(T)$. 
   Note that $\frak{u}^\pm=\bigoplus_{\alpha\in\triangle(\frak{u}^\pm)}\frak{g}_\alpha$ due to \eqref{eq-2.15}, that $\Phi_{[w]}$ is a closed subsystem of $\triangle$ for any $[w]\in\mathscr{W}$, and that $\mathscr{W}_1$ is a Weyl group of $L_\mathbb{C}$.
   Hereafter, let us assume that $\mathscr{W}_1$ is a subgroup of $\mathscr{W}$ via $N_{L_u}(i\frak{h}_\mathbb{R})/C_{L_u}(i\frak{h}_\mathbb{R})\ni\tau C_{L_u}(i\frak{h}_\mathbb{R})\mapsto\tau C_{G_u}(i\frak{h}_\mathbb{R})\in N_{G_u}(i\frak{h}_\mathbb{R})/C_{G_u}(i\frak{h}_\mathbb{R})$. 
   Now, we are in a position to state the proposition: 
\begin{proposition}[{cf.\ Kostant \cite[p.359--361]{Ko1}, \cite[p.121]{Ko2}}]\label{prop-2.17}
   In the setting above$;$ 
\begin{enumerate}
\item[{\rm (1)}]
   For any $[w]\in\mathscr{W}$, it follows that $\triangle^+=\Phi_{[w]}\cup\Phi_{[w\kappa]}$ $($disjoint union$)$, where $[\kappa]$ is the unique element of $\mathscr{W}$ such that $\zeta([\kappa])\triangle^-=\triangle^+$.
\item[{\rm (2)}]
   If $[\sigma]\in\mathscr{W}^1$, then $\zeta([\sigma])^{-1}\bigl(\triangle^+(\frak{l}_\mathbb{C})\bigr)\subset\triangle^+$ and $\zeta([\sigma])^{-1}\bigl(\triangle^-(\frak{l}_\mathbb{C})\bigr)\subset\triangle^-$. 
\item[{\rm (3)}] 
   For each $[w]\in\mathscr{W}$ there exists a unique $([\tau],[\sigma])\in\mathscr{W}_1\times\mathscr{W}^1$ such that $[w]=[\tau\sigma]$.   
\item[{\rm (4)}] 
   For a $[\sigma]\in\mathscr{W}^1$, the following items {\rm (4.i)} and {\rm (4.ii)} hold$:$
\begin{enumerate}
\item[{\rm (4.i)}]
   $n_{[\sigma]}=0$ if and only if $[e]=[\sigma]$.
\item[{\rm (4.ii)}]
   $n_{[\sigma]}=1$ if and only if there exists a $\beta\in\Pi_\triangle$ satisfying $\beta(T)\neq 0$ and $[w_\beta]=[\sigma]$.
\end{enumerate}   
\end{enumerate}
   Here $n_{[\sigma]}$ is the cardinal number of the set $\Phi_{[\sigma]}$, and $e$ is the unit element of $G_\mathbb{C}$.
\end{proposition}   

   Proposition \ref{prop-2.17} enables us to establish 
\begin{proposition}\label{prop-2.18}
   With the same notation and setting as in Proposition {\rm \ref{prop-2.17};} let $r:=\dim_\mathbb{C}\frak{u}^+$.
\begin{enumerate}
\item[{\rm (1)}]
   For each $[\sigma]\in\mathscr{W}^1$ we set 
\[
\begin{array}{ll}
   \triangle_\sigma:=\{\gamma\in\Phi_{[\sigma^{-1}\kappa]} \,|\, \zeta([\sigma])\gamma\in\triangle(\frak{u}^+)\},
   & U^+_\sigma:=\exp\bigl(\bigoplus_{\gamma\in\triangle_\sigma}\frak{g}_{\zeta([\sigma])\gamma}\bigr).
\end{array}   
\]
   Then, $U^+_\sigma$ is a simply connected closed complex nilpotent subgroup of $U^+$ and it is biholomorphic to the $(r-n_{[\sigma]})$-dimensional complex Euclidean space$;$ furthermore, $N^+\sigma^{-1}Q^-=\sigma^{-1}U^+_\sigma Q^-$.
\item[{\rm (2)}] 
   For a $[\sigma]\in\mathscr{W}^1$, the following items {\rm (2.i)} and {\rm (2.ii)} hold$:$
   \begin{enumerate}
   \item[{\rm (2.i)}]
      $\dim_\mathbb{C}U^+_\sigma=r=\dim_\mathbb{C}U^+$ if and only if $[e]=[\sigma]$.
   \item[{\rm (2.ii)}]
      $\dim_\mathbb{C}U^+_\sigma=r-1$ if and only if there exists a $\beta\in\Pi_\triangle$ satisfying $\beta(T)\neq 0$ and $[w_\beta]=[\sigma]$. 
   \end{enumerate}
\item[{\rm (3)}]
   $G_\mathbb{C}=\bigcup_{[\sigma]\in\mathscr{W}^1}N^+\sigma^{-1}Q^-=\bigcup_{[\sigma]\in\mathscr{W}^1}\sigma^{-1}U^+_\sigma Q^-$ $($disjoint unions$)$.
\end{enumerate} 
\end{proposition}   
\begin{proof}
   (1) We only prove that $\dim_\mathbb{C}U^+_\sigma=r-n_{[\sigma]}$ and $N^+\sigma^{-1}Q^-=\sigma^{-1}U^+_\sigma Q^-$ for any $[\sigma]\in\mathscr{W}^1$. 
   In view of \eqref{eq-2.16}, $\zeta([\kappa])\triangle^-=\triangle^+$ and $\triangle(\frak{u}^+)\subset\triangle^+$ we see   
\allowdisplaybreaks{
\begin{align*}
   \zeta([\sigma])\bigl(\triangle_\sigma\bigr)
  &=\{\zeta([\sigma])\gamma\in\triangle(\frak{u}^+) \,|\, \gamma\in\Phi_{[\sigma^{-1}\kappa]}\}\\
  &=\{\zeta([\sigma])\gamma\in\triangle(\frak{u}^+) \,|\, \gamma\in\triangle^+, \zeta([\sigma^{-1}\kappa])^{-1}\gamma\in\triangle^-\}\\
  &=\{\zeta([\sigma])\gamma\in\triangle(\frak{u}^+) \,|\, \gamma\in\triangle^+\}\\
  &=\{\zeta([\sigma])\gamma\in\triangle(\frak{u}^+) \,|\, \zeta([\sigma])^{-1}\bigl(\zeta([\sigma])\gamma\bigr)\in\triangle^+\}
   =\triangle(\frak{u}^+)-\Phi_{[\sigma]}.
\end{align*}}This implies that $\triangle_\sigma$ consists of $(r-n_{[\sigma]})$-elements, so that $\dim_\mathbb{C}U^+_\sigma=r-n_{[\sigma]}$. 
   The Proposition \ref{prop-2.17}-(1) above and Lemma 6.2 in Kostant \cite[p.124]{Ko2} yield 
\[
\begin{array}{rl}
   N^+\sigma^{-1}Q^- 
   & =\exp\bigl(\bigoplus_{\gamma\in\Phi_{[\sigma^{-1}\kappa]}}\frak{g}_\gamma\oplus\bigoplus_{\beta\in\Phi_{[\sigma^{-1}]}}\frak{g}_\beta\bigr)\sigma^{-1}Q^-\\
   & =\exp\bigl(\bigoplus_{\gamma\in\Phi_{[\sigma^{-1}\kappa]}}\frak{g}_\gamma\bigr)\exp\bigl(\bigoplus_{\beta\in\Phi_{[\sigma^{-1}]}}\frak{g}_\beta\bigr)\sigma^{-1}Q^-\\
   & =\sigma^{-1}\exp\bigl(\bigoplus_{\gamma\in\Phi_{[\sigma^{-1}\kappa]}}\frak{g}_{\zeta([\sigma])\gamma}\bigr)\exp\bigl(\bigoplus_{\beta\in\Phi_{[\sigma^{-1}]}}\frak{g}_{\zeta([\sigma])\beta}\bigr)Q^-\\
   & =\sigma^{-1}\exp\bigl(\bigoplus_{\gamma\in\Phi_{[\sigma^{-1}\kappa]}}\frak{g}_{\zeta([\sigma])\gamma}\bigr)Q^-\\
   & =\sigma^{-1}\exp\bigl(\bigoplus_{\gamma\in\triangle_\sigma}\frak{g}_{\zeta([\sigma])\gamma}\bigr)Q^- 
     =\sigma^{-1}U^+_\sigma Q^-,
\end{array}
\]
where we remark that $\bigoplus_{\beta\in\Phi_{[\sigma^{-1}]}}\frak{g}_{\zeta([\sigma])\beta}\subset\frak{n}^-\subset\frak{q}^-$, and that either $\zeta([\sigma])\gamma\in\triangle^+(\frak{l}_\mathbb{C})$ or $\zeta([\sigma])\gamma\in\triangle(\frak{u}^+)$ holds for every $\gamma\in\Phi_{[\sigma^{-1}\kappa]}$; besides, the above computation is independent of the choice of representative $\sigma\in[\sigma]$.\par

   (2) comes from (1) and Proposition \ref{prop-2.17}-(4).\par

   (3) The arguments below will be similar to those in the proof of Lemma 5.6 in Takeuchi \cite[p.21]{Ta} or Proposition 6.1 in Kostant \cite[p.123]{Ko2}.
   
   By virtue of (1), it suffices to confirm that $G_\mathbb{C}=\bigcup_{[\sigma]\in\mathscr{W}^1}N^+\sigma^{-1}Q^-$ (disjoint union).
   Setting $B^+:=N_{G_\mathbb{C}}(\frak{b}^+)$ we fix the Bruhat decomposition $G_\mathbb{C}=\bigcup_{[w]\in\mathscr{W}}N^+w^{-1}B^+$ (disjoint union).
   Then, it follows from $\zeta([\kappa])\triangle^-=\triangle^+$ that $G_\mathbb{C}=\kappa^{-1}G_\mathbb{C}=\bigcup_{[w]\in\mathscr{W}}N^-(w\kappa)^{-1}B^+=\bigcup_{[w]\in\mathscr{W}}N^-w^{-1}B^+$, namely 
\begin{equation}\label{eq-2.19}
   \mbox{$G_\mathbb{C}=\bigcup_{[w]\in\mathscr{W}}N^-w^{-1}B^+$ (disjoint union)}.
\end{equation}
   In a similar way we have 
\[
   \mbox{$L_\mathbb{C}=\bigcup_{[\tau]\in\mathscr{W}_1}N^-_1\tau^{-1}B^+_1$},
\] 
where $\frak{n}^\pm_1:=\bigoplus_{\alpha\in\triangle^\pm(\frak{l}_\mathbb{C})}\frak{g}_\alpha$, $N^-_1:=\exp\frak{n}^-_1$ and $B^+_1:=N_{L_\mathbb{C}}(\frak{h}_\mathbb{C}\oplus\frak{n}^+_1)$.
   This, together with $Q^+=L_\mathbb{C}U^+$ and $B^+=B^+_1U^+$, assures that for any $[\sigma]\in\mathscr{W}^1$, 
\[
\begin{array}{rl}
   N^-\sigma^{-1}Q^+
  &=N^-\sigma^{-1}L_\mathbb{C}U^+
   =\bigcup_{[\tau]\in\mathscr{W}_1}N^-\sigma^{-1}N^-_1\tau^{-1}B^+_1U^+\\
  &=\bigcup_{[\tau]\in\mathscr{W}_1}N^-\sigma^{-1}N^-_1\tau^{-1}B^+
   =\bigcup_{[\tau]\in\mathscr{W}_1}N^-(\tau\sigma)^{-1}B^+,
\end{array} 
\]   
where $\sigma^{-1}N^-_1\subset N^-\sigma^{-1}$ follows from $[\sigma]\in\mathscr{W}^1$ and Proposition \ref{prop-2.17}-(2).
   Consequently, \eqref{eq-2.19} and Proposition \ref{prop-2.17}-(3) allow us to assert that 
\[
   \mbox{$G_\mathbb{C}=\bigcup_{[\sigma]\in\mathscr{W}^1}N^-\sigma^{-1}Q^+$ (disjoint union)}.
\]
   Thus $G_\mathbb{C}=\bigcup_{[\sigma]\in\mathscr{W}^1}N^+\sigma^{-1}Q^-$ (disjoint union) because of $\bar{\theta}(G_\mathbb{C})=G_\mathbb{C}$, $\bar{\theta}(N^-)=N^+$, $\bar{\theta}(\sigma)=\sigma$ and $\bar{\theta}(Q^+)=Q^-$.    
\end{proof} 

   The following corollary will play a role later (recall \eqref{eq-2.13} for $w_\beta$):  
\begin{corollary}\label{cor-2.20}
   Let $G_\mathbb{C}$ be a connected complex semisimple Lie group, let $G$ be a connected closed subgroup of $G_\mathbb{C}$ such that $\frak{g}$ is a real form of $\frak{g}_\mathbb{C}$, and let $T$ be a non-zero elliptic element of $\frak{g}$. 
   Set $U^+$, $Q^-$ as \eqref{eq-2.2}, fix a Cartan decomposition $\frak{g}=\frak{k}\oplus\frak{p}$ with $T\in\frak{k}$, and take a maximal torus $i\frak{h}_\mathbb{R}$ of $\frak{g}_u=\frak{k}\oplus i\frak{p}$ containing $T$.   
   Consider the root system $\triangle=\triangle(\frak{g}_\mathbb{C},\frak{h}_\mathbb{C})$ and a fundamental root system $\Pi_\triangle$ of $\triangle$ such that {\rm (s1)} $\alpha(-iT)\geq 0$ for all $\alpha\in\Pi_\triangle$. 
   Now, let
\[
   \mbox{$\mathcal{O}:=U^+Q^-\cup\bigl(\bigcup_{\mbox{\scriptsize{$\beta\in\Pi_\triangle$ with $\beta(T)\neq0$}}}w_\beta^{-1}U^+Q^-\bigr)$}.
\] 
   Then, $\mathcal{O}$ is a dense domain in $G_\mathbb{C}$. 
   Furthermore, any holomorphic function $f$ on $\mathcal{O}$ can be continued analytically to the whole $G_\mathbb{C}$. 
\end{corollary} 
\begin{proof}
   Proposition \ref{prop-2.7}-(3) implies that $\mathcal{O}$ is a dense domain in $G_\mathbb{C}$.\par
   
   Proposition \ref{prop-2.18} tells us that $e^{-1}U^+_eQ^-\cup\bigl(\bigcup_{\mbox{\scriptsize{$\beta\in\Pi_\triangle$ with $\beta(T)\neq0$}}}w_\beta^{-1}U^+_{w_\beta}Q^-\bigr)$ is a subset of $\mathcal{O}$, and moreover, $G_\mathbb{C}-\mathcal{O}$ must be of complex codimension $2$ or more. 
   Therefore any holomorphic function $f$ on $\mathcal{O}$ can be continued analytically to the whole $G_\mathbb{C}$, by the second Riemann removable singularity theorem (which is sometimes called Hartogs's continuation theorem).
   Here $\dim_\mathbb{C}G_\mathbb{C}\geq 3$, since $G_\mathbb{C}$ is complex semisimple.
\end{proof}

\subsection{Homogeneous holomorphic vector bundles}\label{subsec-2.4}
   In this subsection we recall elementary facts about homogeneous holomorphic vector bundles. 
   
   Let $G_\mathbb{C}$ be a connected complex semisimple Lie group, let $G$ be a connected closed subgroup of $G_\mathbb{C}$ such that $\frak{g}$ is a real form of $\frak{g}_\mathbb{C}$, and let $T$ be a non-zero elliptic element of $\frak{g}$. 
   Define the closed subgroups $L\subset G$ and $Q^-\subset G_\mathbb{C}$ by \eqref{eq-2.2}. 
   Then, we assume that the elliptic orbit $G/L$ is a domain in the complex flag manifold $G_\mathbb{C}/Q^-$ and is a homogeneous complex manifold of $G$ via $\iota:G/L\to G_\mathbb{C}/Q^-$, $gL\mapsto gQ^-$. 
   Now, for a complex vector space ${\sf V}$ of $\dim_\mathbb{C}{\sf V}<\infty$ and a holomorphic homomorphism $\rho:Q^-\to GL({\sf V})$, $q\mapsto\rho(q)$, we denote by $G_\mathbb{C}\times_\rho{\sf V}$ the fiber bundle over $G_\mathbb{C}/Q^-$, with standard fiber ${\sf V}$ and structure group $Q^-$, which is associated to the principal fiber bundle $\pi_\mathbb{C}:G_\mathbb{C}\to G_\mathbb{C}/Q^-$, $x\mapsto xQ^-$, and denote by $\iota^\sharp(G_\mathbb{C}\times_\rho{\sf V})$ the restriction of the bundle $G_\mathbb{C}\times_\rho{\sf V}$ to the domain $G/L\subset G_\mathbb{C}/Q^-$. 
   In this setting, one may assume that 
\begin{gather}
\!   \mathcal{V}_{G_\mathbb{C}/Q^-}\!\!
   :=\!\!\left\{\begin{array}{@{}l@{\,\,}|l@{}}
   h:G_\mathbb{C}\to{\sf V} 
   & \begin{array}{@{\!}l@{\!}}
      \mbox{(1) $h$ is holomorphic},\\
      \mbox{(2) $h(xq)=\rho(q)^{-1}(h(x))$ for all $(x,q)\in G_\mathbb{C}\!\times\! Q^-$}
     \end{array}\end{array}\right\}\!,\!\!\!\!\label{eq-2.21}\\
\!  \mathcal{V}_{G/L}\!\!
   :=\!\!\left\{\begin{array}{@{}l@{\,\,}|l@{}}
   \psi:GQ^-\to{\sf V} 
   & \begin{array}{@{\!}l@{\!}}
      \mbox{(1) $\psi$ is holomorphic},\\
      \mbox{(2) $\psi(yq)=\rho(q)^{-1}(\psi(y))$ for all $(y,q)\in GQ^-\!\times\! Q^-$}
     \end{array}\end{array}\right\}\!\!\!\!\!\!\label{eq-2.22}
\end{gather} 
are the complex vector spaces of holomorphic cross-sections of the bundles $G_\mathbb{C}\times_\rho{\sf V}$ and $\iota^\sharp(G_\mathbb{C}\times_\rho{\sf V})$, respectively. 

\begin{remark}\label{rem-2.23}
   One knows that $\dim_\mathbb{C}\mathcal{V}_{G_\mathbb{C}/Q^-}<\infty$ because $G_\mathbb{C}/Q^-$ is a connected compact complex manifold.
   e.g.\ Kodaira \cite[p.161]{Kod}.
\end{remark}

   From now on, we are going to set a topology for the $\mathcal{V}_{G/L}$. 
   Since $G_\mathbb{C}$ is connected, it satisfies the second countability axiom. 
   Hence $GQ^-$ satisfies the same axiom also and is a locally compact Hausdorff space, since $GQ^-$ is open in $G_\mathbb{C}$. 
   Consequently there exist non-empty open subsets $O_n\subset GQ^-$ such that (i) $GQ^-=\bigcup_{n=1}^\infty O_n$ (countable union) and (ii) the closure $\overline{O_n}$ in $GQ^-$ is compact for each $n\in\mathbb{N}$. 
   Taking a norm $\|\cdot\|$ on the vector space ${\sf V}$, we define $d_n$ by $d_n(\psi_1,\psi_2):=\sup\{\|\psi_1(a)-\psi_2(a)\| : a\in\overline{O_n}\}$ for $n\in\mathbb{N}$, $\psi_1,\psi_2\in\mathcal{V}_{G/L}$; and furthermore we define
\begin{equation}\label{eq-2.24}
   \mbox{$\displaystyle{d(\psi_1,\psi_2):=\sum_{n=1}^\infty\dfrac{1}{2^n}\dfrac{d_n(\psi_1,\psi_2)}{1+d_n(\psi_1,\psi_2)}}$ for $\psi_1,\psi_2\in\mathcal{V}_{G/L}$}.
\end{equation}
   This $d$ is called the {\it Fr\'{e}chet metric} on $\mathcal{V}_{G/L}$. 
   Then, one can show the lemma below (e.g.\ refer to \cite[Paragraph 2.4.4]{BoNo}), where $\varrho:G\to GL(\mathcal{V}_{G/L})$, $g\mapsto\varrho(g)$, is a homomorphism defined by    
\begin{equation}\label{eq-2.25}
   \mbox{$\bigl(\varrho(g)\psi\bigr)(y):=\psi(g^{-1}y)$ for $\psi\in\mathcal{V}_{G/L}$, $y\in GQ^-$}.
\end{equation}

\begin{lemma}\label{lem-2.26}
   In the setting \eqref{eq-2.22}, \eqref{eq-2.24} and \eqref{eq-2.25}$;$ the following four items hold for the Fr\'{e}chet metric $d$ on $\mathcal{V}_{G/L}:$
\begin{enumerate}
\item[{\rm (1)}] 
   $(\mathcal{V}_{G/L},d)$ is a complete metric space.
\item[{\rm (2)}] 
   The metric topology for $(\mathcal{V}_{G/L},d)$ coincides with the topology of uniform convergence on compact sets$;$ and besides it also coincides with the locally convex topology determined by a countable number of seminorms $\{p_n\}_{n\in\mathbb{N}}$, where $p_n(\psi):=d_n(\psi,0)$ for $n\in\mathbb{N}$, $\psi\in\mathcal{V}_{G/L}$. 
\item[{\rm (3)}] 
   Both $\mathcal{V}_{G/L}\times\mathcal{V}_{G/L}\ni(\psi_1,\psi_2)\mapsto\psi_1+\psi_2\in\mathcal{V}_{G/L}$ and $\mathbb{C}\times\mathcal{V}_{G/L}\ni(\alpha,\psi)\mapsto\alpha\psi\in\mathcal{V}_{G/L}$ are continuous mappings.
\item[{\rm (4)}]  
   $G\times\mathcal{V}_{G/L}\ni(g,\psi)\mapsto\varrho(g)\psi\in\mathcal{V}_{G/L}$ is a continuous mapping.
\end{enumerate} 
\end{lemma}

   Lemma \ref{lem-2.26} implies that $\mathcal{V}_{G/L}=(\mathcal{V}_{G/L},d)$ is a Fr\'{e}chet space and that $\varrho$ is a continuous representation of the Lie group $G$ on $\mathcal{V}_{G/L}$. 
   Therefore 
\begin{proposition}[{e.g.\ van den Ban \cite[p.24]{va}}]\label{prop-2.27}
   In the setting \eqref{eq-2.22}, \eqref{eq-2.24} and \eqref{eq-2.25}$;$ for a compact subgroup $K'$ of $G$ we set
\[
   (\mathcal{V}_{G/L})_{K'}:=\big\{\varphi\in\mathcal{V}_{G/L} \,\big|\, \dim_\mathbb{C}\mathrm{span}_\mathbb{C}\{\varrho(k)\varphi : k\in K'\}<\infty\big\}.
\] 
   Then, $(\mathcal{V}_{G/L})_{K'}$ is dense in $\mathcal{V}_{G/L}$ with respect to the metric topology for $(\mathcal{V}_{G/L},d)$.
\end{proposition}

   We end this section with the following remark: the set $(\mathcal{V}_{G/L})_{K'}$ in Proposition \ref{prop-2.27} accords with the set of $K'$-{\it finite vectors} in $\mathcal{V}_{G/L}$ for the continuous representation $\varrho$ of $G$ on $\mathcal{V}_{G/L}$.

\section{The main result in this paper (Theorem \ref{thm-3.1})}\label{sec-3}
   This section consists of two subsections. 
   In Subsection \ref{subsec-3.1} we state Theorem \ref{thm-3.1} which is the main result in this paper; and in Subsection \ref{subsec-3.2} we demonstrate the theorem.
\subsection{The statement of Theorem $\ref{thm-3.1}$}\label{subsec-3.1}
   The setting of Theorem \ref{thm-3.1} is as follows: 
\begin{itemize}
\item
   $G_\mathbb{C}$ is a connected complex semisimple Lie group, 
\item
   $G$ is a connected closed subgroup of $G_\mathbb{C}$ such that $\frak{g}$ is a real form of $\frak{g}_\mathbb{C}$, 
\item
   $T$ is a non-zero elliptic element of $\frak{g}$,
\item 
   $\frak{g}=\frak{k}\oplus\frak{p}$ is a Cartan decomposition of $\frak{g}$ with $T\in\frak{k}$, 
\item 
   $i\frak{h}_\mathbb{R}$ is a maximal torus of $\frak{g}_u=\frak{k}\oplus i\frak{p}$ containing $T$,   
\item 
   $\triangle=\triangle(\frak{g}_\mathbb{C},\frak{h}_\mathbb{C})$ is the root system of $\frak{g}_\mathbb{C}$ relative to $\frak{h}_\mathbb{C}$, where $\frak{h}_\mathbb{C}$ is the complex vector subspace of $\frak{g}_\mathbb{C}$ generated by $i\frak{h}_\mathbb{R}$,
\item 
   $\frak{g}_\alpha$ is the root subspace of $\frak{g}_\mathbb{C}$ for $\alpha\in\triangle$,  
\item 
   $L=C_G(T)$, 
\item 
   $Q^-$ is the closed complex subgroup of $G_\mathbb{C}$ defined by \eqref{eq-2.2},
\item 
   $\frak{k}_\mathbb{C}$ is the complex subalgebra of $\frak{g}_\mathbb{C}$ generated by $\frak{k}$,
\item
   ${\sf V}$ is a finite dimensional complex vector space,
\item
   $\rho:Q^-\to GL({\sf V})$, $q\mapsto\rho(q)$, is a holomorphic homomorphism, 
\item
   $\mathcal{V}_{G_\mathbb{C}/Q^-}$ and $\mathcal{V}_{G/L}$ are the complex vector spaces defined by \eqref{eq-2.21} and \eqref{eq-2.22}, respectively. 
\end{itemize}   
   Now, we are in a position to state 
\begin{theorem}\label{thm-3.1}
   In the setting of Subsection $\ref{subsec-3.1};$ suppose that {\rm (S)} there exists a fundamental root system $\Pi_\triangle$ of $\triangle$ satisfying  
\begin{enumerate}
\item[]{\rm (s1)}  
   $\alpha(-iT)\geq 0$ for all $\alpha\in\Pi_\triangle$, and
\item[]{\rm (s2)}  
   $\frak{g}_\beta\subset\frak{k}_\mathbb{C}$ for every $\beta\in\Pi_\triangle$ with $\beta(T)\neq0$.
\end{enumerate}
   Then, the complex vector space $\mathcal{V}_{G_\mathbb{C}/Q^-}$ is linear isomorphic to $\mathcal{V}_{G/L}$ via $F:\mathcal{V}_{G_\mathbb{C}/Q^-}\to\mathcal{V}_{G/L}$, $h\mapsto h|_{GQ^-}$, and therefore $\dim_\mathbb{C}\mathcal{V}_{G/L}=\dim_\mathbb{C}\mathcal{V}_{G_\mathbb{C}/Q^-}<\infty$. 
   Here $h|_{GQ^-}$ stands for the restriction of $h$ to $GQ^-\subset G_\mathbb{C}$.
\end{theorem} 

\subsection{Proof of Theorem $\ref{thm-3.1}$}\label{subsec-3.2}
   The setting of Theorem \ref{thm-3.1} remains valid in this subsection.
   In addition, we take the closed complex subgroup $U^+$ defined by \eqref{eq-2.2} and the maximal compact subgroup $K\subset G$ corresponding to the subalgebra $\frak{k}\subset\frak{g}$ into consideration.\par 

   Our goal in Subsection \ref{subsec-3.2} is to complete the proof of Theorem \ref{thm-3.1}. 
   We are going to show two Lemmas \ref{lem-3.5} and \ref{lem-3.6}, Proposition \ref{prop-3.7} and Corollary \ref{cor-3.10}, and obtain the goal from them.
  
\begin{remark}\label{rem-3.2}
    For the element $T$ concerning Theorem \ref{thm-3.1}, one may assume that 
\begin{equation}\label{eq-3.3}
   \mbox{all the eigenvalues of $\mathrm{ad}iT$ are integer}.
\end{equation}
   Let us explain the reason why. 
   Let $T'$ be the element in Lemma \ref{lem-2.6}. 
   Then for any root $\alpha\in\triangle$, one can assert that $\alpha(iT)>0$, $\alpha(iT)=0$ and $\alpha(iT)<0$ if and only if $\alpha(iT')>0$, $\alpha(iT')=0$ and $\alpha(iT')<0$, respectively. 
   In particular, for any $\beta\in\Pi_\triangle$, $\beta(T)\neq0$ if and only if $\beta(T')\neq0$.
   Accordingly there are no changes in the topological group structures on $L$ and $Q^-$, and no change in the supposition (S) even if one substitutes $T'$ for $T$. 
   For this reason, we assume \eqref{eq-3.3} hereafter.\par
 
   Suppose that $\triangle(\frak{u}^+)$ consists of $r$-elements $\gamma_1,\gamma_2,\dots,\gamma_r\in\triangle$, where $r=\dim_\mathbb{C}\frak{u}^+$. 
   Then, $\{E_{\gamma_j}\}_{j=1}^r$ is a complex base of $\frak{u}^+=\bigoplus_{\alpha\in\triangle(\frak{u}^+)}\frak{g}_\alpha=\bigoplus_{j=1}^r\frak{g}_{\gamma_j}$, and by virtue of \eqref{eq-3.3} there exist $n_1,n_2,\dots,n_r\in\mathbb{N}$ satisfying $\gamma_j(T)=in_j$ for each $1\leq j\leq r$, so that 
\begin{equation}\label{eq-3.4}
   \mbox{$\mathrm{Ad}(\exp\lambda T)E_{\gamma_j}=e^{in_j\lambda}E_{\gamma_j}$ ($1\leq j\leq r$)}
\end{equation}
for all $\lambda\in\mathbb{R}$.
   cf.\ \eqref{eq-2.10} for $E_{\gamma_j}$, \eqref{eq-2.16} for $\triangle(\frak{u}^+)$.
\end{remark}

\begin{lemma}\label{lem-3.5}
   The mapping $F:\mathcal{V}_{G_\mathbb{C}/Q^-}\to\mathcal{V}_{G/L}$, $h\mapsto h|_{GQ^-}$, is injective linear.
\end{lemma}
\begin{proof}
   It is enough to confirm that $F$ is injective. 
   That comes from the theorem of identity, since $h:G_\mathbb{C}\to{\sf V}$ is holomorphic, $G_\mathbb{C}$ is connected and $GQ^-$ is open in $G_\mathbb{C}$.
\end{proof}

\begin{lemma}\label{lem-3.6}
\begin{enumerate}
\item[]
\item[{\rm (1)}]
   Let $\varphi$ be a $K$-finite vector in $\mathcal{V}_{G/L}$ for the representation $\varrho$ defined by \eqref{eq-2.25}, and let $\mathcal{V}_\varphi$ be the complex vector subspace of $\mathcal{V}_{G/L}$ generated by $\{\varrho(k)\varphi : k\in K\}$. 
   Then, there exist a complex base $\{\varphi_a\}_{a=1}^{m_\varphi}$ of $\mathcal{V}_\varphi$ and $\mu_1,\mu_2,\dots,\mu_{m_\varphi}\in\mathbb{R}$ such that 
\[
   \mbox{$\varrho(\exp\lambda T)\varphi_a=e^{i\mu_a\lambda}\varphi_a$}
\]
for all $1\leq a\leq{m_\varphi}=\dim_\mathbb{C}\mathcal{V}_\varphi$ and $\lambda\in\mathbb{R}$. 
\item[{\rm (2)}]
   There exist a complex base $\{{\sf v}_b\}_{b=1}^k$ of ${\sf V}$ and $\theta_1,\theta_2,\dots,\theta_k\in\mathbb{R}$ such that 
\[
   \mbox{$\rho(\exp\lambda T){\sf v}_b=e^{i\theta_b\lambda}{\sf v}_b$}
\]
for all $1\leq b\leq k=\dim_\mathbb{C}{\sf V}$ and $\lambda\in\mathbb{R}$. 
\end{enumerate}
\end{lemma}
\begin{proof}
   (1) Let $S^1:=\{\exp\lambda T \,|\, \lambda\in\mathbb{R}\}$. 
   Then, it follows from $T\in\frak{k}$ that $S^1$ is a real one-dimensional torus and $S^1\subset K$. 
   Therefore, since $\mathcal{V}_\varphi$ is $\varrho(K)$-invariant and $m_\varphi=\dim_\mathbb{C}\mathcal{V}_\varphi<\infty$, there exist $\varrho(S^1)$-invariant complex vector subspaces $\mathcal{V}_1,\mathcal{V}_2,\dots,\mathcal{V}_{m_\varphi}\subset\mathcal{V}_\varphi$ and $\mu_1,\mu_2,\dots,\mu_{m_\varphi}\in\mathbb{R}$ such that $\mathcal{V}_\varphi=\mathcal{V}_1\oplus\mathcal{V}_2\oplus\cdots\oplus\mathcal{V}_{m_\varphi}$, $\dim_\mathbb{C}\mathcal{V}_a=1$ and 
\[
   \mbox{$\varrho(\exp\lambda T)=e^{i\mu_a\lambda}\mathrm{id}$ on $\mathcal{V}_a$}
\]
for all $1\leq a\leq m_\varphi$ and $\lambda\in\mathbb{R}$. 
   Hence we can get the conclusion by taking a non-zero element of $\mathcal{V}_a$ for each $1\leq a\leq m_\varphi$.\par

   (2) One can conclude (2) by arguments similar to those above. 
   Indeed; there exist $\rho(S^1)$-invariant complex vector subspaces ${\sf V}_1,\dots,{\sf V}_k\subset{\sf V}$ and $\theta_1,\dots,\theta_k\in\mathbb{R}$ such that ${\sf V}={\sf V}_1\oplus\cdots\oplus{\sf V}_k$, $\dim_\mathbb{C}{\sf V}_b=1$ and $\rho(\exp\lambda T)=e^{i\theta_b\lambda}\mathrm{id}$ on ${\sf V}_b$ for all $1\leq b\leq k$ and $\lambda\in\mathbb{R}$, because of $S^1\subset Q^-$ and $k=\dim_\mathbb{C}{\sf V}<\infty$.
\end{proof}

\begin{proposition}\label{prop-3.7}
   Let $\{\varphi_a\}_{a=1}^{m_\varphi}$ and $\{{\sf v}_b\}_{b=1}^k$ be the complex bases of $\mathcal{V}_\varphi$ and ${\sf V}$ in Lemma {\rm \ref{lem-3.6}}, respectively. 
   For $y\in GQ^-$ we express $\varphi_a(y)\in{\sf V}$ as 
\[
   \varphi_a(y)=\varphi_a^1(y){\sf v}_1+\varphi_a^2(y){\sf v}_2+\cdots+\varphi_a^k(y){\sf v}_k.
\]  
   Then, for each $1\leq b\leq k$ there exists a unique polynomial $($holomorphic$)$ function $\varphi_a^b{}'$ on $\mathbb{C}^r\cong U^+$ of finite degree such that 
\[
   \mbox{$\varphi_a^b=\varphi_a^b{}'|_{U^+\cap GQ^-}$}.
\]
   Therefore, for a given $\phi\in\mathcal{V}_\varphi$ there exists a unique holomorphic mapping $\phi':U^+\to{\sf V}$ such that $\phi=\phi'|_{U^+\cap GQ^-}$. 
\end{proposition}
\begin{proof}
   Denote by $z^1,z^2,\dots,z^r$ the canonical coordinates of the first kind associated with the complex base $\{E_{\gamma_j}\}_{j=1}^r$ of $\frak{u}^+$ (see Remark \ref{rem-3.2} for $E_{\gamma_j}$). 
   Here, it turns out that $U^+\cong\mathbb{C}^r$ via 
\[
   U^+\ni\exp(z^1E_{\gamma_1}+z^2E_{\gamma_2}+\cdots+z^rE_{\gamma_r})\mapsto(z^1,z^2,\dots,z^r)\in\mathbb{C}^r.
\]
   Noting that $U^+\cap GQ^-$ is an open subset of $U^+$ containing the unit $e\in G_\mathbb{C}$ and that the restriction $\varphi_a^b|_{U^+\cap GQ^-}$ is a holomorphic function on $U^+\cap GQ^-$, we obtain an $R>0$ so that the following (i) and (ii) hold for $O:=\{u\in U^+ : \mbox{$|z^j(u)|<R$, $1\leq j\leq r$}\}$:
\begin{enumerate}
\item[(i)]  
   $O$ is an open subset of $U^+\cap GQ^-$ containing $e$, and 
\item[(ii)]
   on $O$ we can express $\varphi_a^b|_{U^+\cap GQ^-}$ as 
\[
   \varphi_a^b(z^1,z^2,\dots,z^r)=\sum_{m_1,m_2,\dots,m_r\geq0}\alpha^b_{m_1m_2\cdots m_r}(z^1)^{m_1}(z^2)^{m_2}\cdots(z^r)^{m_r}
\]
(the Taylor expansion of $\varphi_a^b|_{U^+\cap GQ^-}$ at $e=(0,0,\dots,0)$).
\end{enumerate}
   Remark that $O$ is stable under every inner automorphism of $S^1=\{\exp\lambda T \,|\, \lambda\in\mathbb{R}\}$, cf.\ \eqref{eq-3.4}. 
   For any $\lambda\in\mathbb{R}$ and $u\in O$ we have 
\[
\begin{split}
    \sum_{b=1}^ke^{i\theta_b\lambda}&\varphi_a^b(u){\sf v}_b
   =\rho(\exp\lambda T)(\sum_{b=1}^k\varphi_a^b(u){\sf v}_b)\quad\mbox{($\because$ Lemma \ref{lem-3.6}-(2))}\\
  &=\rho(\exp\lambda T)(\varphi_a(u))
   =\varphi_a(u\exp(-\lambda T)) \quad\mbox{($\because$ $\exp\lambda T\in Q^-$, \eqref{eq-2.22}-(2))}\\
  &=\bigl(\varrho(\exp\lambda T)\varphi_a\bigr)\bigl((\exp\lambda T)u\exp(-\lambda T)\bigr) \quad\mbox{($\because$ $\exp\lambda T\in G$, \eqref{eq-2.25})}\\
  &=(e^{i\mu_a\lambda}\varphi_a)\bigl((\exp\lambda T)u\exp(-\lambda T)\bigr)\quad\mbox{($\because$ Lemma \ref{lem-3.6}-(1))}\\
  &=\sum_{b=1}^ke^{i\mu_a\lambda}\varphi_a^b\bigl((\exp\lambda T)u\exp(-\lambda T)\bigr){\sf v}_b,
\end{split}
\]
and hence $e^{i\theta_b\lambda}\varphi_a^b(u)=e^{i\mu_a\lambda}\varphi_a^b\bigl((\exp\lambda T)u\exp(-\lambda T)\bigr)$. 
   This, together with (ii) and \eqref{eq-3.4}, yields  
\[
\begin{split}
  &\sum_{m_1,m_2,\dots,m_r\geq0}e^{i(\theta_b-\mu_a)\lambda}\alpha^b_{m_1m_2\cdots m_r}(z^1)^{m_1}(z^2)^{m_2}\cdots(z^r)^{m_r}\\
  &\qquad =e^{i(\theta_b-\mu_a)\lambda}\varphi_a^b(z^1,z^2,\dots,z^r)=e^{i(\theta_b-\mu_a)\lambda}\varphi_a^b(u)\\
  &\qquad =\varphi_a^b\bigl((\exp\lambda T)u\exp(-\lambda T)\bigr)
   =\varphi_a^b(e^{in_1\lambda}z^1,e^{in_2\lambda}z^2,\dots,e^{in_r\lambda}z^r)\\
  &\qquad =\sum_{m_1,m_2,\dots,m_r\geq0}e^{i(n_1m_1+n_2m_2+\dots+n_rm_r)\lambda}\alpha^b_{m_1m_2\cdots m_r}(z^1)^{m_1}(z^2)^{m_2}\cdots(z^r)^{m_r} 
\end{split}
\]
in case of $u=\exp(z^1E_{\gamma_1}+z^2E_{\gamma_2}+\cdots+z^rE_{\gamma_r})$.
   Therefore one shows that  
\[
   e^{i(\theta_b-\mu_a)\lambda}\alpha^b_{m_1m_2\cdots m_r}=e^{i(n_1m_1+n_2m_2+\dots+n_rm_r)\lambda}\alpha^b_{m_1m_2\cdots m_r}.
\]   
   Differentiating this equation at $\lambda=0$, we deduce that  
\begin{equation}\label{eq-3.8}
   (\theta_b-\mu_a)\alpha^b_{m_1m_2\cdots m_r}=(n_1m_1+n_2m_2+\dots+n_rm_r)\alpha^b_{m_1m_2\cdots m_r}
\end{equation}
for all $1\leq b\leq k$ and $m_1,m_2,\dots,m_r\geq0$.
   It follows from $n_1,n_2,\dots,n_r\in\mathbb{N}$ and \eqref{eq-3.8} that for each $b$, all the coefficients $\alpha^b_{m_1m_2\cdots m_r}$ vanish whenever $\theta_b-\mu_a\not\in\mathbb{N}\cup\{0\}$; and that with respect to $m_1,m_2,\dots,m_r$ with $\theta_b-\mu_a\neq n_1m_1+n_2m_2+\dots+n_rm_r$, the coefficient $\alpha^b_{m_1m_2\cdots m_r}$ vanishes even if $\theta_b-\mu_a\in\mathbb{N}\cup\{0\}$. 
   These imply that 
\[
   \varphi_a^b(z^1,z^2,\dots,z^r)=\sum_{m_1,m_2,\dots,m_r\geq0}\alpha^b_{m_1m_2\cdots m_r}(z^1)^{m_1}(z^2)^{m_2}\cdots(z^r)^{m_r}
\]
must be a polynomial function on $O$ of finite degree. 
   Consequently, for each $1\leq b\leq k$, $\varphi_a^b(z^1,\dots,z^r)$ can extend uniquely to a polynomial function $\varphi_a^b{}'(z^1,\dots,z^r)$ on $\mathbb{C}^r\cong U^+$ of finite degree.    
\end{proof}

\begin{remark}\label{rem-3.9}
   In Proposition \ref{prop-3.7} we have concluded that for any $\phi\in\mathcal{V}_\varphi$, the restriction $\phi|_{U^+\cap GQ^-}$ can be continued analytically to $U^+$, without the supposition (s2) in Theorem \ref{thm-3.1}.
\end{remark}

   Proposition \ref{prop-3.7} leads to 
\begin{corollary}\label{cor-3.10}
   Let $\varphi$ be any $K$-finite vector in $\mathcal{V}_{G/L}$ for the representation $\varrho$ defined by \eqref{eq-2.25}, and let $\mathcal{V}_\varphi$ be the complex vector subspace of $\mathcal{V}_{G/L}$ generated by $\{\varrho(k)\varphi : k\in K\}$.
   Suppose that {\rm (S)} there exists a fundamental root system $\Pi_\triangle$ of $\triangle$ satisfying  
\begin{enumerate}
\item[]{\rm (s1)} 
   $\alpha(-iT)\geq 0$ for all $\alpha\in\Pi_\triangle$, and
\item[]{\rm (s2)} 
   $\frak{g}_\beta\subset\frak{k}_\mathbb{C}$ for every $\beta\in\Pi_\triangle$ with $\beta(T)\neq0$.
\end{enumerate}
   Then, it follows that $\varphi\in\mathcal{V}_\varphi\subset F(\mathcal{V}_{G_\mathbb{C}/Q^-})$.
\end{corollary}
\begin{proof} 
   Take any $\phi\in\mathcal{V}_\varphi$. 
   By Proposition \ref{prop-3.7} there exists a unique holomorphic mapping $\phi':U^+\to{\sf V}$ such that $\phi=\phi'|_{U^+\cap GQ^-}$.
   Proposition \ref{prop-2.7}-(3) enables us to construct the holomorphic extension $\phi'':U^+Q^-\to{\sf V}$ of $\phi'$ from 
\[
   \mbox{$\phi''(uq):=\rho(q)^{-1}\bigl(\phi'(u)\bigr)$ for $(u,q)\in U^+\times Q^-$}.
\]  
   Here, it follows from $(U^+Q^-\cap GQ^-)=(U^+\cap GQ^-)Q^-$, $\phi=\phi'|_{U^+\cap GQ^-}$, $\phi\in\mathcal{V}_{G/L}$ and \eqref{eq-2.22}-(2) that 
\[
   \mbox{$\phi=\phi''$ on $U^+Q^-\cap GQ^-$}.
\]
   Now, Lemma \ref{lem-2.14} and (s2) assure that $w_\beta\in K$ for every $\beta\in\Pi_\triangle$ with $\beta(T)\neq0$. 
   This enables us to obtain 
\[
   \varrho(w_\beta)\phi\in\mathcal{V}_\varphi,
\]
since $\mathcal{V}_\varphi$ is $\varrho(K)$-invariant. 
   Accordingly for each $\beta\in\Pi_\triangle$ with $\beta(T)\neq0$, there exists a unique holomorphic mapping $(\varrho(w_\beta)\phi)'':U^+Q^-\to{\sf V}$ such that 
\[
   \mbox{$\varrho(w_\beta)\phi=(\varrho(w_\beta)\phi)''$ on $U^+Q^-\cap GQ^-$}.
\]
   Then, we define a holomorphic mapping $\hat{\phi}$ from 
\[
   \mbox{$\mathcal{O}=U^+Q^-\cup\bigl(\bigcup_{\mbox{\scriptsize{$\beta\in\Pi_\triangle$ with $\beta(T)\neq0$}}}w_\beta^{-1}U^+Q^-\bigr)$}
\]
into ${\sf V}$ as follows: 
\begin{equation}\label{eq-3.11}
   \hat{\phi}(x):=\begin{cases} 
   \phi''(x) & \mbox{if $x\in U^+Q^-$},\\
   (\varrho(w_\beta)\phi)''(w_\beta x) & \mbox{if $x\in w_\beta^{-1}U^+Q^-$}.
   \end{cases}
\end{equation}
   Here $\mathcal{O}$ is a dense, domain in $G_\mathbb{C}$ (cf.\ Corollary \ref{cor-2.20}). 
   Let us confirm that the definition \eqref{eq-3.11} is well-defined.
   Corollary \ref{cor-2.9}-(1) implies that the intersection 
\[
   \mbox{$GQ^-\cap U^+Q^-\cap\bigl(\bigcap_{\mbox{\scriptsize{$\beta\in\Pi_\triangle$ with $\beta(T)\neq0$}}}w_\beta^{-1}U^+Q^-\bigr)$}
\]
is a non-empty open subset of $G_\mathbb{C}$. 
   For any element $y$ of the intersection above and any $\beta\in\Pi_\triangle$ with $\beta(T)\neq0$ we have $w_\beta y\in U^+Q^-$ and $w_\beta y\in KGQ^-\subset GQ^-$; and thus   
\[
   (\varrho(w_\beta)\phi)''(w_\beta y)
   =(\varrho(w_\beta)\phi)(w_\beta y)
   \stackrel{\eqref{eq-2.25}}{=}\phi(y)
   =\phi''(y)
\]
in terms of $w_\beta y,y\in U^+Q^-\cap GQ^-$.
   For this reason \eqref{eq-3.11} is well-defined by the theorem of identity and it follows that $\phi=\hat{\phi}$ on $\mathcal{O}\cap GQ^-$. 
   From Corollary \ref{cor-2.20}, there exists the analytic continuation $\hat{\phi}':G_\mathbb{C}\to{\sf V}$ of $\hat{\phi}:\mathcal{O}\to{\sf V}$. 
   This $\hat{\phi}'$ satisfies $\hat{\phi}'(xq)=\rho(q)^{-1}(\hat{\phi}'(x))$ for all $(x,q)\in G_\mathbb{C}\times Q^-$, by the theorem of identity, $\phi=\hat{\phi}'|_{GQ^-}$, \eqref{eq-2.22}-(2) and $\phi\in\mathcal{V}_{G/L}$.  
   Consequently it is immediate from \eqref{eq-2.21} that $\hat{\phi}'\in\mathcal{V}_{G_\mathbb{C}/Q^-}$, so that $\phi=\hat{\phi}'|_{GQ^-}=F(\hat{\phi}')\in F(\mathcal{V}_{G_\mathbb{C}/Q^-})$. 
   This provides us with $\mathcal{V}_\varphi\subset F(\mathcal{V}_{G_\mathbb{C}/Q^-})$.
\end{proof}

   Now, let us demonstrate Theorem \ref{thm-3.1}.
\begin{proof}[Proof of Theorem {\rm \ref{thm-3.1}}]
   By Lemma \ref{lem-3.5} and Remark \ref{rem-2.23} it suffices to conclude  
\begin{equation}\label{eq-3.12}
   \mathcal{V}_{G/L}\subset F(\mathcal{V}_{G_\mathbb{C}/Q^-}).
\end{equation}
   Let $(\mathcal{V}_{G/L})_K$ be the set of $K$-finite vectors in $\mathcal{V}_{G/L}$ for the representation $\varrho$ defined by \eqref{eq-2.25}. 
   From Corollary \ref{cor-3.10} we obtain
\begin{equation}\label{eq-3.13}
   (\mathcal{V}_{G/L})_K\subset F(\mathcal{V}_{G_\mathbb{C}/Q^-}).
\end{equation}
   Now, let $\psi$ be an arbitrary element of $\mathcal{V}_{G/L}$. 
   On the one hand; Proposition \ref{prop-2.27} assures that there exists a sequence $\{\varphi_n\}_{n=1}^\infty\subset(\mathcal{V}_{G/L})_K$ satisfying 
\[
   \lim_{n\to\infty}d(\psi,\varphi_n)=0.
\] 
   On the other hand; since $\mathcal{V}_{G/L}=(\mathcal{V}_{G/L},d)$ is a Hausdorff topological vector space and $\dim_\mathbb{C}F(\mathcal{V}_{G_\mathbb{C}/Q^-})=\dim_\mathbb{C}\mathcal{V}_{G_\mathbb{C}/Q^-}<\infty$, it turns out that $F(\mathcal{V}_{G_\mathbb{C}/Q^-})$ is closed in $\mathcal{V}_{G/L}$. 
   Thus, it follows from \eqref{eq-3.13} that $\psi=\lim_{n\to\infty}\varphi_n\in F(\mathcal{V}_{G_\mathbb{C}/Q^-})$, so that \eqref{eq-3.12} holds. 
\end{proof}

\section{Examples}\label{sec-4}
   Let us give some examples which satisfy the supposition (S) in Theorem \ref{thm-3.1} and an example which does not so. 
   Recall that the supposition is as follows: 
\begin{quote}
   {\rm (S)} there exists a fundamental root system $\Pi_\triangle$ of $\triangle$ satisfying  
   \begin{enumerate}
   \item[]{\rm (s1)} $\alpha(-iT)\geq 0$ for all $\alpha\in\Pi_\triangle$, and
   \item[]{\rm (s2)} $\frak{g}_\beta\subset\frak{k}_\mathbb{C}$ for every $\beta\in\Pi_\triangle$ with $\beta(T)\neq0$.
   \end{enumerate}
\end{quote}

\begin{example}[{$G/L=SU(p,q)/S(U(h)\times U(p-h,q))$, $p+q\geq 2$, $0<h<p$}]\label{ex-4.1}
   Let $G_\mathbb{C}:=SL(p+q,\mathbb{C})$, $G:=SU(p,q)$, $\frak{g}_u:=\frak{su}(p+q)$ and  
\[
   \frak{h}_\mathbb{R}:=\left\{\begin{array}{@{}c@{\,\,}|@{\,\,}c@{}}
   \begin{pmatrix} x_1 & & \huge{O} \\  & \ddots & \\ \huge{O} & & x_{p+q}\end{pmatrix}
   & \displaystyle{x_l\in\mathbb{R},\, \sum_{l=1}^{p+q}x_l=0}\end{array}\right\},
\]
where $p+q\geq 2$.
   Denote by $\triangle=\triangle(\frak{g}_\mathbb{C},\frak{h}_\mathbb{C})$ the root system of $\frak{g}_\mathbb{C}$ relative $\frak{h}_\mathbb{C}$, define simple roots $\alpha_k\in\triangle$ ($1\leq k\leq p+q-1$) as
\[
   \alpha_k\Big(\begin{pmatrix} z_1 & & \huge{O} \\  & \ddots & \\ \huge{O} & & z_{p+q}\end{pmatrix}\Big):=z_k-z_{k+1},
\] 
and set $\Pi_\triangle:=\{\alpha_k\}_{k=1}^{p+q-1}$. 
   Here, the dual base $\{Z_k\}_{k=1}^{p+q-1}$ of $\Pi_\triangle=\{\alpha_k\}_{k=1}^{p+q-1}$ is 
\[
   \mbox{$Z_k=\dfrac{1}{p+q}\begin{pmatrix} (p+q-k)I_k & \huge{O}\\ \huge{O} & -kI_{p+q-k}\end{pmatrix}$ for $1\leq k\leq p+q-1$},
\]
where $I_n$ is the unit matrix of degree $n$. 
   Let $T_h:=iZ_h$, $0<h<p$, and   
\[
\begin{split}
&  \frak{k}:=\!\left\{\begin{array}{@{}c@{\,\,}|@{\,\,}c@{}}
   \begin{pmatrix} A_p & \huge{O} \\ \huge{O} & D_q\end{pmatrix}\!\in\frak{g}
   & \mbox{$A_p$ : $p\times p$ matrix, $D_q$ : $q\times q$ matrix}\end{array}\right\}\!,\\
&  \frak{p}:=\!\left\{\begin{array}{@{}c@{\,\,}|@{\,\,}c@{}}
   \begin{pmatrix} \huge{O} & B_{p\times q} \\ C_{q\times p} & \huge{O}\end{pmatrix}\!\in\frak{g}
   & \mbox{$B_{p\times q}$ : $p\times q$ matrix, $C_{q\times p}$ : $q\times p$ matrix}\end{array}\right\}\!.
\end{split}   
\] 
   In the setting above, it follows that $T_h$ is an elliptic element of $\frak{g}$, $i\frak{h}_\mathbb{R}$ is a maximal torus of $\frak{g}_u$ containing $T_h$, $\frak{g}=\frak{k}\oplus\frak{p}$, $\frak{g}_u=\frak{k}\oplus i\frak{p}$ and (s1) $\alpha(-iT_h)\geq 0$ for all $\alpha\in\Pi_\triangle$.  
   Moreover, 
\begin{enumerate}
\item 
   for $\beta\in\Pi_\triangle=\{\alpha_k\}_{k=1}^{p+q-1}$, $\beta(T_h)\neq0$ if and only if $\beta=\alpha_h$, 
\item 
   $\frak{g}_{\alpha_h}=\mathrm{span}_\mathbb{C}\{E_{h,h+1}\}$,    
\end{enumerate}
where $\frak{g}_{\alpha_h}$ is the root subspace of $\frak{g}_\mathbb{C}$ for $\alpha_h$ and $E_{h,h+1}$ is the matrix whose $(h,h+1)$-element is $1$ and whose other elements are all $0$. 
   Since $0<h<p$, we have (s2) 
\[
   \frak{g}_{\alpha_h}\subset\frak{k}_\mathbb{C}.
\]  
   For this reason, the supposition (S) in Theorem \ref{thm-3.1} holds for this example. 
   Incidentally, $L=C_G(T_h)=S(U(h)\times U(p-h,q))$, and Theorem \ref{thm-3.1} implies that the complex Lie algebra $\mathcal{O}(T^{1,0}(G/L))$ of holomorphic vector fields on $G/L=SU(p,q)/S(U(h)\times U(p-h,q))$ is isomorphic to $\frak{sl}(p+q,\mathbb{C})$, where $p+q\geq 2$, $0<h<p$. 
\end{example}

   Unfortunately, there are examples of elliptic orbits to which we cannot apply Theorem \ref{thm-3.1}.
\begin{example}\label{ex-4.2}
   The supposition (S) in Theorem \ref{thm-3.1} cannot hold for any symmetric bounded domain $D$ in $\mathbb{C}^n$ at all.\par

   Let us explain the reason why. 
   In order to do so, we consider an elliptic orbit $G/L=G/C_G(T)$ in the setting of Subsection \ref{subsec-3.1}, and put $\frak{u}:=[T,\frak{g}]$. 
   Since $\mathrm{ad}T\in\mathrm{End}(\frak{g})$ is semisimple and $\frak{l}=\frak{c}_\frak{g}(T)$ one can decompose $\frak{g}$ as $\frak{g}=\frak{l}\oplus\frak{u}$, and furthermore decompose it as follows: 
\[
   \frak{g}=(\frak{k}\cap\frak{l})\oplus(\frak{p}\cap\frak{l})\oplus(\frak{k}\cap\frak{u})\oplus(\frak{p}\cap\frak{u})
\]  
because of $T\in\frak{k}$. 
   Then, Lemma \ref{lem-2.14} tells us that  
\[
   \frak{k}\cap\frak{u}\neq\{0\}
\]
is a necessary condition for the (s2) to hold. 
   However, if $G/L$ is a symmetric bounded domain in $\mathbb{C}^n$ (where $G$ is the identity component of $\mathrm{Hol}(G/L)$), then $\frak{k}\cap\frak{l}=\frak{k}$, $\frak{p}\cap\frak{l}=\{0\}$, $\frak{k}\cap\frak{u}=\{0\}$ and $\frak{p}\cap\frak{u}=\frak{p}$. 
   For this reason, the supposition (S) cannot hold for the $D$ at all.  
\end{example}

   The following example is interesting, we think: 
\begin{example}[$G/L=G_{2(2)}/(SL(2,\mathbb{R})\cdot T^1)$]\label{ex-4.3}
   Let $\frak{g}_\mathbb{C}$ be the exceptional complex simple Lie algebra $(\frak{g}_2)_\mathbb{C}$ of the type $G_2$. 
   Assume that the Dynkin diagram of $\triangle=\triangle(\frak{g}_\mathbb{C},\frak{h}_\mathbb{C})$ is as follows (cf.\ Bourbaki \cite[p.289]{Br}\footnote{There is a minor misprint in \cite{Br}: p.289, $\downarrow$ 9, Add $\alpha_2$ to (II) Positive roots.}):
\begin{center}
\unitlength=1mm
\begin{picture}(19,11) 
 \put(9,1){$3$} 
 \put(7,8){$\alpha_1$} 
 \put(9,6){\circle{2}}  
 \put(10,6){\line(1,1){4}}
 \put(10,6){\line(1,-1){4}}   
 \put(10,6.5){\line(1,0){5}} 
 \put(10,6){\line(1,0){5}}
 \put(10,5.5){\line(1,0){5}}  
 \put(16,1){$2$} 
 \put(16,8){$\alpha_2$} 
 \put(16,6){\circle{2}}
\put(1,5.5){$\frak{g}_\mathbb{C}$:} 
\end{picture}
\end{center}
   First of all, let us set a non-compact real form $\frak{g}$ of $\frak{g}_\mathbb{C}$.
   Define a compact real form $\frak{g}_u$ of $\frak{g}_\mathbb{C}$ by $\frak{h}_\mathbb{R}:=\mathrm{span}_\mathbb{R}\{H_\alpha\,|\,\alpha\in\triangle\}$, $\frak{g}_u:=i\frak{h}_\mathbb{R}\oplus\bigoplus_{\alpha\in\triangle}\mathrm{span}_\mathbb{R}\{E_\alpha-E_{-\alpha}\}\oplus\mathrm{span}_\mathbb{R}\{i(E_\alpha+E_{-\alpha})\}$, and denote by $\{Z_1,Z_2\}\subset\frak{h}_\mathbb{R}$ the dual base of $\Pi_\triangle=\{\alpha_1,\alpha_2\}$ (cf.\ Paragraph \ref{subsec-2.3.1} for $H_\alpha$, $E_\alpha$). 
   By use of this $Z_2$ we set 
\begin{equation}\label{eq-4.4}
   \theta:=\exp\pi\mathrm{ad}(iZ_2).
\end{equation} 
   Then $\theta$ is an involutive automorphism of the complex Lie algebra $\frak{g}_\mathbb{C}$ such that $\theta(\frak{g}_u)\subset\frak{g}_u$, and we define a non-compact real form $\frak{g}\subset\frak{g}_\mathbb{C}$ in the following way:
\[
\begin{array}{lll} 
   \frak{k}:=\{X\in\frak{g}_u \,|\, \theta(X)=X\},
   & i\frak{p}:=\{Y\in\frak{g}_u \,|\, \theta(Y)=-Y\}, 
   & \frak{g}:=\frak{k}\oplus\frak{p}.
\end{array}
\]
   Remark here that $\frak{g}_u=\frak{k}\oplus i\frak{p}$, $\frak{k}=\frak{sp}(1)\oplus\frak{sp}(1)$ and $\frak{g}=\frak{g}_{2(2)}$; besides, 
\[
   \frak{k}_\mathbb{C}=\{Z\in\frak{g}_\mathbb{C} \,|\, \theta(Z)=Z\},
\]
where $\frak{k}_\mathbb{C}$ is the complex subalgebra of $\frak{g}_\mathbb{C}$ generated by $\frak{k}$.
\begin{center}
\unitlength=1mm
\begin{picture}(33,7) 
 \put(8,4){$\alpha_1$} 
 \put(9,2){\circle{2}}
 \put(15,4){$-3\alpha_1-2\alpha_2$} 
 \put(22,2){\circle{2}}  
\put(1,2){$\frak{k}_\mathbb{C}$:} 
\end{picture}
\end{center}
   In this setting, a given $T\in i\frak{h}_\mathbb{R}$ is an elliptic element of $\frak{g}$ and we know that for $\frak{l}:=\frak{c}_\frak{g}(T)$,
\begin{enumerate}
\item[(a)]
   $\frak{l}=\frak{sl}(2,\mathbb{R})\oplus\frak{t}^1$ in case of $T=i(Z_1-2Z_2)$, 
\item[(b)]
   $\frak{l}=\frak{sl}(2,\mathbb{R})\oplus\frak{t}^1$ in case of $T=i(Z_1-3Z_2)$.
\end{enumerate} 
   cf.\ Proposition 5.5 \cite[p.1157]{Bo}. 
   We investigate the cases (a) and (b), individually.\par

   Case (a): Let $T:=i(Z_1-2Z_2)$ and $\Pi_a:=\{2\alpha_1+\alpha_2,-3\alpha_1-2\alpha_2\}$. 
   Then $\Pi_a$ is a fundamental root system of $\triangle$ such that (s1) $\alpha(-iT)\geq 0$ for all $\alpha\in\Pi_a$. 
   Indeed, it follows from $\alpha_k(Z_j)=\delta_{kj}$ that $(2\alpha_1+\alpha_2)(-iT)=0$ and $(-3\alpha_1-2\alpha_2)(-iT)=1$.  
\begin{center}
\unitlength=1mm
\begin{picture}(29,11) 
 \put(7,8){$2\alpha_1+\alpha_2$} 
 \put(15,6){\circle{2}}  
 \put(16,6){\line(1,1){4}}
 \put(16,6){\line(1,-1){4}}   
 \put(16,6.5){\line(1,0){5}} 
 \put(16,6){\line(1,0){5}}
 \put(16,5.5){\line(1,0){5}}   
 \put(19,1){$-3\alpha_1-2\alpha_2$} 
 \put(22,6){\circle{2}}
\put(1,5){$\Pi_a$:} 
\end{picture}
\end{center}
   Since \eqref{eq-4.4} yields $\theta(E_{-3\alpha_1-2\alpha_2})=E_{-3\alpha_1-2\alpha_2}$, we have (s2) $\frak{g}_{-3\alpha_1-2\alpha_2}\subset\frak{k}_\mathbb{C}$. 
   Therefore the supposition (S) in Theorem \ref{thm-3.1} holds in this case.\par

   Case (b): Let $T:=i(Z_1-3Z_2)$ and $\Pi_b:=\{\alpha_1,-3\alpha_1-\alpha_2\}$. 
   Then, $\Pi_b$ is a fundamental root system of $\triangle$ such that (s1) $\alpha_1(-iT)=1$ and $(-3\alpha_1-\alpha_2)(-iT)=0$.        
\begin{center}
\unitlength=1mm
\begin{picture}(17,11) 
 \put(7,8){$\alpha_1$} 
 \put(10,6){\circle{2}}  
 \put(11,6){\line(1,1){4}}
 \put(11,6){\line(1,-1){4}}   
 \put(11,6.5){\line(1,0){5}} 
 \put(11,6){\line(1,0){5}}
 \put(11,5.5){\line(1,0){5}}   
 \put(14,1){$-3\alpha_1-\alpha_2$} 
 \put(17,6){\circle{2}}  
\put(1,5){$\Pi_b$:} 
\end{picture}
\end{center} 
   From \eqref{eq-4.4} one obtains (s2) $\theta(E_{\alpha_1})=E_{\alpha_1}$. 
   Hence the supposition (S) in Theorem \ref{thm-3.1} holds in this case, also. 
\end{example}

   We end this paper with a comment on Example \ref{ex-4.3}, $G/L=G_{2(2)}/(SL(2,\mathbb{R})\cdot T^1)$. 
   In both the cases (a) and (b), the supposition (S) in Theorem \ref{thm-3.1} holds. 
   So, in each case Theorem \ref{thm-3.1} implies that the complex Lie algebra $\mathcal{O}(T^{1,0}(G/L))$ of holomorphic vector fields on $G/L$ is isomorphic to $\mathcal{O}(T^{1,0}(G_\mathbb{C}/Q^-))$. 
   Then, 
\begin{enumerate}
\item[a.]
   $\mathcal{O}(T^{1,0}(G/L))$ is isomorphic to $(\frak{g}_2)_\mathbb{C}$ in case (a); but, in contrast, 
\item[b.]
   $\mathcal{O}(T^{1,0}(G/L))$ is isomorphic to $\frak{so}(7,\mathbb{C})$ in case (b).
\end{enumerate} 
   cf.\ the proof of Theorem 7.1 in Oni\v{s}\v{c}ik \cite[p.238--239]{On}. 

\section*{Acknowledgements}
   The author would like to express his sincere gratitude to Professor Soji Kaneyuki for the valuable suggestions in Kyoto, 24 October 2005.


\end{document}